%% file: __main__.tex
\pdfoutput=1 
\documentclass[12pt]{article}

\usepackage[T1]{fontenc}

\usepackage[utf8]{inputenc}

\usepackage[english]{babel}

\usepackage[nomath]{lmodern}

\usepackage{amsmath}
\usepackage{amssymb}
\usepackage{amsthm}

\usepackage{stmaryrd}

\usepackage{geometry}

\usepackage{enumitem}

\usepackage{tikz}
\usetikzlibrary{arrows.meta} 
\usetikzlibrary{graphs} 
\usetikzlibrary{decorations.pathreplacing} 

\usepackage{hyperref}
\hypersetup{
	pdftitle = {Compactifiable classes of compacta},
	pdfauthor = {A. Bartoš, J. Bobok, J. van Mill, P. Pyrih, B. Vejnar},
	linkbordercolor = {0.5 0.5 1},
	citebordercolor = {0.5 1 0.5}
}
\date{May 2, 2019}

\theoremstyle{definition}

\newtheorem{theorem}{Theorem}[section]
\newtheorem{definition}[theorem]{Definition}
\newtheorem{notation}[theorem]{Notation}
\newtheorem{construction}[theorem]{Construction}
\newtheorem{proposition}[theorem]{Proposition}
\newtheorem{observation}[theorem]{Observation}
\newtheorem{corollary}[theorem]{Corollary}

\newtheorem{remark}[theorem]{Remark}
\newtheorem{lemma}[theorem]{Lemma}
\newtheorem{example}[theorem]{Example}

\newtheorem{question}[theorem]{Question}
\newtheorem*{acknowledgements}{Acknowledgements}

\newcommand*\cdef{\newcommand*}

\setlist{itemsep = 0pt}
\setlist[enumerate]{leftmargin=*} 
\setlist[enumerate, 1]{label=\upshape (\roman*), ref=(\roman*)}
\setlist[enumerate, 2]{label=\upshape (\alph*), ref=(\alph*)}

\cdef \DescriptionFormat [1]{%
	\normalfont\emph{#1}%
}

\setlist[description]{format=\DescriptionFormat}

\cdef \Show [1]{%
	\expandafter\show\csname #1\endcsname
}

\cdef \AlignedMath [1]{%
	\begingroup
	
	\begin{align*}#1\end{align*}%
	\endgroup
	\ignorespaces
}

\let \O \undefined 
\let \P \undefined 
\let \H \undefined 

\cdef \customfootnote [2]{%
	\begingroup
	\footnotetext{\textsuperscript{#1}#2}%
	\endgroup
}

\cdef \alt [2]{%
	#1 (resp.\ #2)%
}

\begin{document}

\input{macros/unicode_characters}
	\input{macros/local_labels}
	\input{macros/notation}

	
	\cdef \complexityFsigma
		{\cite[Theorem~3.6]{Bartos}}
	
	\cdef \complexityClosed
		{\cite[Observation~4.3]{Bartos}}

	\title{Compactifiable classes of compacta}
	\author{%
		A. Bartoš\textsuperscript{*,1,2}, 
		J. Bobok\textsuperscript{3}, 
		J. van Mill\textsuperscript{4},
		P. Pyrih\textsuperscript{1}, 
		B. Vejnar\textsuperscript{1}
	}
	\customfootnote{*}{e-mail: \texttt{drekin@gmail.com}}
	\customfootnote{1}{Charles University, Faculty of Mathematics and Physics, Department of Mathematical Analysis}
	\customfootnote{2}{Czech Academy of Sciences, Institute of Mathematics, Department of Abstract Analysis}
	\customfootnote{3}{Czech Technical University in Prague, Faculty of Civil Engineering, Department of Mathematics}
	\customfootnote{4}{University of Amsterdam, Faculty of Science, Korteweg–de Vries Institute for Mathematics}
	
	\maketitle
	
	\vspace{-1.5em}
	\begin{center} \em
		Dedicated to the memory of Petr Simon, \\ member of Seminar on Topology at Charles University.
	\end{center}
	\vspace{0em} 
	
	\begin{abstract}
		We introduce the notion of compactifiable classes – these are classes of metrizable compact spaces that can be up to homeomorphic copies “disjointly combined” into one metrizable compact space.
		This is witnessed by so-called compact composition of the class.
		Analogously, we consider Polishable classes and Polish compositions.
		The question of compactifiability or Polishability of a class is related to hyperspaces.
		Strongly compactifiable and strongly Polishable classes may be characterized by the existence of a corresponding family in the hyperspace of all metrizable compacta.
		We systematically study the introduced notions – we give several characterizations, consider preservation under various constructions, and raise several questions.
		
		\begin{description}
			\item[Classification:]
				54D80, 
				54H05, 
				54B20, 
				54E45, 
				54F15. 
			
			\item[Keywords:] Compactifiable class, Polishable class, homeomorphism equivalence, metrizable compactum, Polish space, hyperspace, complexity, universal element, common model, inverse limit.
		\end{description}
	\end{abstract}
	
	\linespread{1.2}\selectfont 

\input{introduction}

\input{compositions}

\input{hyperspaces}

\input{induced_classes}

\input{inverse_limits}

	\begin{acknowledgements}
		The paper is an outgrowth of joint research of the Open Problem Seminar, Charles University, Prague.
		A. Bartoš was supported by the grant projects GAUK 970217 and SVV-2017-260456 of Charles University, 
		and by the grant project GA17-27844S of Czech Science Foundation (GAČR) with institutional support RVO 67985840.
		J. Bobok was supported by the European Regional Development Fund, project No.~CZ.02.1.01/0.0/0.0/16\_019/0000778.
		B. Vejnar was supported by Charles University Research Centre program No.UNCE/SCI/022 and GAČR 17-00941S.
	\end{acknowledgements}
	
	\linespread{1}\selectfont
	
	\bibliographystyle{siam}
	\bibliography{references}
	
\end{document}

%% file: macros/unicode_characters.tex
\cdef \DefUnicode [2]{%
	\expandafter\cdef
		\csname u8:\detokenize{#1}\endcsname
		{#2}%
}

\cdef \UndefUnicode [1]{%
	\expandafter\let
		\csname u8:\detokenize{#1}\endcsname
		\undefined
}

\cdef \ShowUnicode [1]{%
	\expandafter\show
		\csname u8:\detokenize{#1}\endcsname
}

\UndefUnicode{ } 

\UndefUnicode{…} 
\DefUnicode{…}{\ldots} 

\DefUnicode{¬}{\lnot}

\DefUnicode{×}{\times}


\DefUnicode{α}{\alpha}
\DefUnicode{β}{\beta}
\DefUnicode{γ}{\gamma}
\DefUnicode{δ}{\delta}
\DefUnicode{ε}{\varepsilon}
\DefUnicode{ζ}{\zeta}
\DefUnicode{η}{\eta}
\DefUnicode{θ}{\theta}
\DefUnicode{ι}{\iota}
\DefUnicode{κ}{\kappa}
\DefUnicode{λ}{\lambda}
\DefUnicode{μ}{\mu}
\DefUnicode{ν}{\nu}
\DefUnicode{ξ}{\xi}
\DefUnicode{π}{\pi}
\DefUnicode{ρ}{\rho}
\DefUnicode{σ}{\sigma}
\DefUnicode{ς}{\varsigma}
\DefUnicode{τ}{\tau}
\DefUnicode{υ}{\upsilon}
\DefUnicode{φ}{\varphi}
\DefUnicode{χ}{\chi}
\DefUnicode{ψ}{\psi}
\DefUnicode{ω}{\omega}

\DefUnicode{Γ}{\Gamma}
\DefUnicode{Δ}{\Delta}
\DefUnicode{Θ}{\Theta}
\DefUnicode{Λ}{\Lambda}
\DefUnicode{Ξ}{\Xi}
\DefUnicode{Π}{\Pi}
\DefUnicode{Σ}{\Sigma}
\DefUnicode{Φ}{\Phi}
\DefUnicode{Ψ}{\Psi}
\DefUnicode{Ω}{\Omega}

\DefUnicode{ℂ}{\CC}
\DefUnicode{ℍ}{\HH}
\DefUnicode{ℕ}{\NN}
\DefUnicode{ℙ}{\PP}
\DefUnicode{ℚ}{\QQ}
\DefUnicode{ℝ}{\RR}
\DefUnicode{ℤ}{\ZZ}

\DefUnicode{←}{\leftarrow}
\DefUnicode{→}{\rightarrow}
\DefUnicode{↛}{\nrightarrow}

\DefUnicode{∀}{\forall}
\DefUnicode{∃}{\exists}
\DefUnicode{∄}{\nexists}
\DefUnicode{∅}{\emptyset}
\DefUnicode{∈}{\in}
\DefUnicode{∉}{\notin}
\DefUnicode{∋}{\owns}
\DefUnicode{∌}{\notowns}
\DefUnicode{∏}{\prod}
\DefUnicode{∐}{\coprod}
\DefUnicode{∑}{\sum}
\DefUnicode{∘}{\circ}
\DefUnicode{∞}{\infty}
\DefUnicode{∧}{\wedge}
\DefUnicode{∨}{\vee}
\DefUnicode{∩}{\cap}
\DefUnicode{∪}{\cup}
\DefUnicode{≠}{\neq}
\DefUnicode{≤}{\leq}
\DefUnicode{≥}{\geq}
\DefUnicode{⊆}{\subseteq}
\DefUnicode{⊇}{\supseteq}
\DefUnicode{⊈}{\nsubseteq}
\DefUnicode{⊉}{\nsupseteq}
\DefUnicode{⊊}{\subsetneq}
\DefUnicode{⊋}{\supsetneq}
\DefUnicode{⋃}{\bigcup}
\DefUnicode{⋂}{\bigcap}

\DefUnicode{∂}{\partial}

\DefUnicode{⋅}{\cdot}

\DefUnicode{⟨}{\langle}
\DefUnicode{⟩}{\rangle}

%% file: macros/notation.tex

\cdef \A {\mathcal{A}}
\cdef \B {\mathcal{B}}
\cdef \C {\mathcal{C}}
\cdef \D {\mathcal{D}}
\cdef \E {\mathcal{E}}
\cdef \F {\mathcal{F}}
\cdef \G {\mathcal{G}}
\cdef \H {\mathcal{H}}
\cdef \K {\mathcal{K}}
\cdef \M {\mathcal{M}}
\cdef \N {\mathcal{N}}
\cdef \O {\mathcal{O}}
\cdef \P {\mathcal{P}}
\cdef \R {\mathcal{R}}
\cdef \U {\mathcal{U}}
\cdef \V {\mathcal{V}}
\cdef \W {\mathcal{W}}
\cdef \Y {\mathcal{Y}}

\cdef \NN {\mathbb{N}}
\cdef \RR {\mathbb{R}}
\cdef \CC {\mathbb{C}}


\cdef \holds {%
	\colon
}

\cdef \st {%
	\colon
}


\cdef \given {%
	\colon
}

\cdef \set [1]{%
	\{#1\}%
}

\cdef \tuple [1]{%
	(#1)%
}


\cdef \card [1]{%
	\lvert #1\rvert
}

\cdef \powset [1]{%
	\mathcal{P}(#1)%
}

\cdef \subsets [2]{%
	[#2]^{#1}%
}

\cdef \disunion {%
	\sqcup
}

\cdef \DisUnion {%
	\bigsqcup
}

\cdef \continuum {%
	\mathfrak{c}%
}


\cdef \len [1]{%
	\lvert #1\rvert
}

\cdef \concat {%
	{}^\frown
}

\cdef \Concat {%
	\bigsqcup
}


\cdef \maps {%
	\colon
}

\cdef \into {%
	\hookrightarrow
}

\cdef \onto {%
	\twoheadrightarrow
}

\cdef \from {%
	\leftarrow
}

\cdef \id {%
	\operatorname{id}%
}

\cdef \dom {%
	\operatorname{dom}%
}

\cdef \rng {%
	\operatorname{rng}%
}

\cdef \cod {%
	\operatorname{cod}%
}

\cdef \im [1]{%
	[#1]%
}

\cdef \inv {%
	^{-1}%
}

\cdef \preim [1]{%
	\inv\im{#1}%
}

\cdef \fiber [1]{%
	\inv(#1)%
}

\cdef \restr [1]{%
	\mathord{\upharpoonright}_{#1}%
}

\cdef \diag {%
	\mathbin\vartriangle
}

\cdef \Diag {%
	\bigtriangleup
}

\cdef \codiag {%
	\mathbin\triangledown
}

\cdef \CoDiag {%
	\bigtriangledown
}


\cdef \homeo {%
	\cong
}

\cdef \nothomeo {%
	\ncong
}

\cdef \clo [2][]{%
	\overline{#2}%
}

\cdef \cl {%
	\operatorname{cl}%
}

\cdef \topsum {%
	\oplus
}

\cdef \TopSum {%
	\sum
}


\cdef \dist {%
	\operatorname{d}%
}

\cdef \diam {%
	\operatorname{diam}%
}

\cdef \mesh {%
	\operatorname{mesh}%
}

\cdef \conv {%
	\rightarrow
}

\cdef \biconv {%
	\leftrightarrow
}

\cdef \abs [1]{%
	\lvert #1\rvert
}


\cdef \Closed [1]{%
	\C l(#1)%
}

\cdef \Compacta [1]{%
	\K(#1)%
}

\cdef \Continua [2][]{%
	\C#1(#2)%
}

\cdef \Cont {%
	C%
}

\cdef \ContU {%
	\C_{\operatorname{u}}%
}

\cdef \ContSurjU {%
	\S_{\operatorname{u}}%
}

\cdef \ev {%
	\operatorname{ev}%
}

\cdef \img {%
	\operatorname{img}%
}

\cdef \ImageMap [1]{%
	#1^*%
}

\cdef \PreimageMap [1]{%
	#1^{-1**}%
}
\cdef \FiberMap [1]{%
	#1^{-1*}%
}


\cdef \Top {%
	\mathbf{Top}%
}

\cdef \HComp {%
	\mathbf{HComp}%
}

\cdef \Obj {%
	\operatorname{Obj}%
}

\cdef \InvLim {%
	\operatorname{Lim}_\leftarrow
}

\cdef \InvSeq {%
	\operatorname{Seq}_\leftarrow
}

\cdef \LimSeq [1]{%
	\operatorname{LimSeq}_\leftarrow(#1)%
}

\cdef \half {%
	\frac{1}{2}%
}

\cdef \upset [1]{%
	#1^{\uparrow}%
}

\cdef \downset [1]{%
	#1^{\downarrow}%
}

\cdef \unions [1]{%
	#1^{⋃}%
}

\cdef \homeocopies [1]{%
	#1^{\homeo}%
}

\cdef \contimages [1]{%
	#1^{\twoheadrightarrow}%
}

\cdef \contpreimages [1]{%
	#1^{\twoheadleftarrow}%
}

\cdef \AllCompacta {%
	\mathbf{K}%
}

\cdef \AllContinua {%
	\mathbf{C}%
}

%% file: introduction.tex
\section{Introduction}
	
	Let us consider two classes $\C$ and $\D$ of topological spaces (not necessarily closed under homeomorphic copies). We say that these classes are \emph{equivalent} (and we write $\C \homeo \D$) if every space in $\C$ is homeomorphic to a space in $\D$ and vice versa.
	
	Given a class $\C$ of metrizable compacta, we are interested whether $\C$ (up to the equivalence) can be disjointly composed into one metrizable compactum such that the corresponding quotient space is also a metrizable compactum.
	In our terminology introduced below, we ask whether the class $\C$ is \emph{compactifiable}.
	If $\C$ is a class of continua, this is equivalent to finding a metrizable compactum whose set of connected components is equivalent to $\C$ (see Observation~\ref{thm:components}).
	
	Original motivation comes from our interest in spirals \cite{BMPV} and from the construction of Minc~\cite{Minc}, who for each nondegenerate metric continuum $X$ constructed a metrizable compactum $K$ whose components form a pairwise non-homeomorphic family of spirals over $X$ with the decomposition space being $2^ω$, and asked \cite[Question~1]{Minc} whether there is a metrizable compactum $K$ whose set of components is equivalent to the class of all spirals over $X$, i.e. whether the class of all spirals over $X$ is compactifiable.
	So compactifiability of a class may be viewed as a dual condition to the existence of a metrizable compactum whose components from a pairwise non-homeomorphic subfamily of the class.
	Minc~\cite[Question~2]{Minc} also asked whether both conditions may be realized at the same time and/or whether the resulting decomposition may be continuous. This latter property corresponds to our notion of \emph{strongly compactifiable classes}.
	
	In Section~\ref{sec:compositions} of our paper we define \emph{compactifiable} and \emph{Polishable} classes and their witnessing \emph{compositions}. We consider several basic constructions of compositions, and we obtain several conditions equivalent to compactifiability and Polishability (Theorem~\ref{thm:compactifiable_characterization} and \ref{thm:Polishable_characterization}).
	
	In Section~\ref{sec:hyperspaces} we study connections between compactifiable or Polishable classes and hyperspaces.
	The Hilbert cube $[0, 1]^ω$ is universal for metrizable compacta, so a class of metrizable compacta may be realized as a subset of the hyperspace $\Compacta{[0, 1]^ω}$.
	We define \emph{strongly compactifiable} and \emph{strongly Polishable} classes, and characterize them by the existence of an equivalent family $\F ⊆ \Compacta{[0, 1]^ω}$ of a suitable complexity – closed or equivalently $F_σ$ for strong compactifiability and $G_δ$ or equivalently analytic for strong Polishability (Theorem~\ref{thm:strongly_compactifiable_characterization} and \ref{thm:strongly_Polishable_characterization}).
	Note that if a class $\C$ closed under homeomorphic copies is strongly compactifiable, $\C ∩ \Compacta{[0, 1]^ω}$ is not necessarily closed – there is only an equivalent closed family $\F ⊆ \Compacta{[0, 1]^ω}$.
	This leads to considering descriptive complexity of subsets of $\Compacta{[0, 1]^ω}$ up to the equivalence.
	The first author further develops this topic in \cite{Bartos}.
	
	In Section~\ref{sec:induced} we study preservation of the properties under various constructions, and consequently we obtain several examples.
	Among other results we prove the following.
	The four introduced properties are stable under countable unions.
	Every hereditary class of metrizable compacta or continua with a universal element is strongly compactifiable,
	and every class of metrizable \alt{compacta}{continua} closed under continuous images with a common model (i.e. a member of the class that continuously maps onto every other member of the class) is \alt{strongly Polishable}{compactifiable}.
	For every strongly Polishable class $\C$ closed under homeomorphic copies and every Polish space $X$, the set $\C ∩ \Compacta{X}$ is analytic – this gives a necessary condition.
	
	We may view the properties of being strongly compactifiable, compactifiable, strongly Polishable, and Polishable as degrees of complexity – classes of metrizable compacta that are compactifiable are “more comprehensible” than classes that are not compactifiable.
	A different measure of complexity of a class $\C$ of metrizable compacta is the complexity of the corresponding classification problem, i.e. the \emph{Borel reducibility}~\cite[Chapter~5]{Gao} of the homeomorphism relation $\homeo_\C$.
	However, we first need to realize $\C$ as a standard Borel space in a natural way, e.g. as a subset of the hyperspace $\Compacta{[0, 1]^ω}$.
	That means this notion formally depends on the choice of such natural coding, even though it is a common belief that the particular natural coding does not matter in fact.
	See for example \cite[Theorem~14.1.3]{Gao}.
	
	Another inspiration for our study was the construction of a universal arc-like continuum \cite[Theorem~12.22]{Nadler}.
	In Section~\ref{sec:limits} we modify this construction and prove that for every countable family $\P$ of metrizable compacta, the class of all $\P$-like spaces is compactifiable.
	We also argue that a compact composition may be viewed as a weaker form of a universal element for the class.
	
	Several questions remain open.
	We do not have any particular example distinguishing between the four properties (Question~\ref{q:distinguish}), we have just some candidates.
	Also, the compactifiability of spirals remains open.

%% file: compositions.tex
\section{Compositions} \label{sec:compositions}
	
	In this section we formally define \emph{compactifiable} and \emph{Polishable classes} and the witnessing \emph{compositions}.
	We describe several constructions of compositions and give some characterizations of compactifiability and Polishability.
	We also observe that compactifiable and Polishable classes are stable under countable unions.
	In particular, every countable class of metrizable compacta is compactifiable.
	
	\medskip
	
	The idea of disjointly composing topological spaces is captured by the following notion.
	
	\begin{definition}
		A \emph{composition} $\A$ consists of a continuous map $q\maps A \to B$ between topological spaces.
		In this context, $A$ is called the \emph{composition space}, $B$ is called the \emph{indexing space}, and $q$ is called the \emph{composition map}.
		The idea is that the composition map $q$ captures how its fibers are composed in the composition space $A$.
		The notation $\A(q\maps A \to B)$ means that $\A$ is a composition with composition space $A$, indexing space $B$, and composition map $q$.
		
		The following language gives us some flexibility when working with compositions.
		\begin{itemize}
			\item $\A$ is a \emph{composition of an indexed family of topological spaces} $\tuple{A_b}_{b ∈ B}$ if $q\fiber{b} = A_b$ for every $b ∈ B$.
			Of course the family $\tuple{A_b}_{b ∈ B}$ is a decomposition of $A$ (i.e.\ $A_b ∩ A_{b'} = ∅$ for every $b ≠ b' ∈ B$ and $⋃_{b ∈ B} A_b = A$) and is determined by $\A$.
			On the other hand, every decomposition $\tuple{A_b}_{b ∈ B}$ of a topological space $A$ induces the unique map $q\maps A \to B$ with fibers $\tuple{A_b}_{b ∈ B}$ and the composition $\A(q\maps A \to B)$ if the map $q$ is continuous.
		
			\item $\A$ is a \emph{composition of an indexed family of embeddings} $\tuple{e_b\maps A_b \into A}_{b ∈ B}$ if $q\fiber{b} = \rng(e_b)$ for every $b ∈ B$.
			Again, $\tuple{\rng(e_b)}_{b ∈ B}$ is necessarily a decomposition of $A$.
			
			\item $\A$ is a \emph{composition of a class of topological spaces} $\C$ if the family $\set{q\fiber{b}: b ∈ B}$ is equivalent to $\C$.
		\end{itemize}
		We are interested in the following special types of compositions.
		\begin{itemize}
			\item $\A$ is a \emph{compact composition} if both $A$ and $B$ are metrizable compacta.
			\item $\A$ is a \emph{Polish composition} if both $A$ and $B$ are Polish spaces.
		\end{itemize}
	\end{definition}
	
	\begin{remark}
		In \cite{Minc} P.\ Minc constructed a compact composition of a $2^ω$-indexed family of pairwise non-homeomorphic compactifications of a ray with remainders being copies of an arbitrary fixed nondegenerate metrizable continuum.
	\end{remark}
	
	\begin{remark}
		Given a composition $\A(q\maps A \to B)$ of a family $\tuple{A_b}_{b ∈ B}$, the spaces $A_b$ are all nonempty if and only if the composition map $q$ is surjective.
	\end{remark}
	
	\begin{definition}
		A class $\C$ of topological spaces is called \alt{\emph{compactifiable}}{\emph{Polishable}} if there is a \alt{compact}{Polish} composition of $\C$, i.e.\ if there is a continuous map $q\maps A \to B$ between \alt{metrizable compacta}{Polish spaces} such that $\set{q\fiber{b}: b ∈ B} \homeo \C$.
		Note that the spaces $q\fiber{b}$ are necessarily metrizable \alt{compacta}{Polish spaces}.
	\end{definition}
	
	\begin{construction}[rectangular composition] \label{con:rectangular}
		Let $A$, $B$ be topological spaces and let $F ⊆ A × B$. By $F^b$ we denote the subset of $A$ corresponding to the section of $F$ through $b$, i.e.\ $F^b = \set{a ∈ A \given \tuple{a, b} ∈ F}$. For every $b ∈ B$ let $e_b$ denote the canonical embedding $F^b \to F^b × \set{b} ⊆ F$.
		The set $F$ induces the composition $\A_F(π_B\restr{F}\maps F \to B)$ of the family $\tuple{e_b}_{b ∈ B}$.
		If the spaces $A$, $B$ are \alt{metrizable compacta}{Polish spaces} and the set $F$ is \alt{closed}{$G_δ$} in $A × B$, then the composition $\A_F$ is \alt{compact}{Polish}.
		
		Moreover, every composition can essentially be obtained this way.
		For a composition $\A(q\maps A \to B)$ we consider the graph of $q$, $G = \set{\tuple{a, q(a)}: a ∈ A} ⊆ A × B$, which is closed if $B$ is Hausdorff.
		Since $A$ is homeomorphic to $G$ and $G^b = q\fiber{b}$ for every $b ∈ B$, the compositions $\A$ and $\A_G$ are essentially the same.
	\end{construction}
	
	\begin{construction}[pullback composition] \label{con:pullback}
		Let $\A(q\maps A \to B)$ be a composition and let $f\maps B' \to B$ be a continuous map.
		The \emph{pullback of $\A$ along $f$} is the composition $\A'(q'\maps A' \to B')$ where $A' := \set{\tuple{a, b'} ∈ A × B': q(a) = f(b')}$ and $q' := π_{B'}\restr{A'}$, so $\A'$ is the rectangular composition induced by $A' ⊆ A × B'$.
		
		If $\A$ is a composition of spaces $\tuple{A_b}_{b ∈ B}$, then $\A'$ is essentially a composition of $\tuple{A_{f(b')}}_{b' ∈ B'}$ since for every $b' ∈ B'$ we have the canonical embedding $e_{b'}\maps A_{f(b')} \to A_{f(b')} × \set{b'} ⊆ A'$ and so $\A'$ is formally a composition of $\tuple{e_{b'}}_{b' ∈ B'}$.
		This way we change the indexing space so that each space $A_b$ has $f\fiber{b}$-many copies in $A'$.
		
		Moreover, $A'$ is a closed subset $A × B'$ if $B$ is Hausdorff.
		Hence, if $\A$ is a \alt{compact}{Polish} composition and $B'$ is \alt{a metrizable compactum}{a Polish space}, then $\A'$ is a \alt{compact}{Polish} composition as well.
	\end{construction}
	
	\begin{corollary}[subcomposition] \label{thm:subcomposition}
		If $\A(q\maps A \to B)$ is a \alt{compact}{Polish} composition of spaces $\tuple{A_b}_{b ∈ B}$ and $C ⊆ B$ is \alt{$F_σ$}{analytic}, then the class $\set{A_c: c ∈ C}$ is \alt{compactifiable}{Polishable}.
		
		\begin{proof}
			In the compact case with closed $C ⊆ B$, it is enough to consider the \emph{induced subcomposition} $\A_C(q\maps q\preim{C} \to C)$, which may be viewed as a special case of the pullback construction.
			If $C = ⋃_{n ∈ ω} C_n$ for some closed sets $C_n ⊆ B$, then $\set{A_c: c ∈ C}$ is a countable union of compactifiable classes, which is compactifiable as we will show later (Observation~\ref{thm:countable_union_of_compositions}).
			In the Polish case, there is a Polish space $B'$ and a continuous surjection $f\maps B' \onto C$, so the pullback of $\A$ along $f$ is a Polish composition of $\set{A_{f(b')}: b' ∈ B'} = \set{A_c: c ∈ C}$.
		\end{proof}
	\end{corollary}
	
	\begin{remark}
		We always consider an analytic set as a subset of a Polish space.
		By \emph{analytic space} we mean any topological space that arises from an analytic set endowed with the corresponding subspace topology, i.e. a metrizable continuous image of a Polish space.
		However, in the following constructions (like in Lemma~\ref{thm:analytic_pullback}) we in fact do not need the metrizability, so the propositions would remain valid even for non-metrizable continuous images of Polish spaces.
	\end{remark}
	
	\begin{lemma} \label{thm:analytic_pullback}
		Let $A$ be a Polish space, let $B$ be an analytic space, let $F ⊆ A × B$ be a $G_δ$ subset, and let $\A_F(q\maps F \to B)$ be the corresponding rectangular composition.
		Moreover, let $B'$ be a Polish space and let $f\maps B' \to B$ be a continuous map.
		The pullback $\A'(q'\maps F' \to B')$ of $\A_F$ along $f$ is a Polish composition.
		
		\begin{proof}
			We need to show that the composition space $F'$ is Polish.
			We have $F' = \set{\tuple{\tuple{a, b}, b'} ∈ (A × B) × B': \tuple{a, b} ∈ F$ and $b = f(b')}$, which is canonically homeomorphic to $G := \set{\tuple{a, b'} ∈ A × B': \tuple{a, f(b')} ∈ F} = g\preim{F}$ where $g := \id_A × f \maps A × B' \to A × B$.
			Since $F$ is $G_δ$ in $A × B$, $G$ is $G_δ$ in the Polish space $A × B'$.
		\end{proof}
	\end{lemma}
	
	By combining the previous observations we obtain the following characterizations.
	
	\begin{theorem} \label{thm:compactifiable_characterization}
		The following conditions are equivalent for a class $\C$ of topological spaces.
		\begin{enumerate}
			\item $\C$ is compactifiable.
			\item There is a metrizable compactum $A$ and a closed equivalence relation $E ⊆ A × A$ such that $\set{E^a: a ∈ A} \homeo \C \setminus \set{∅}$.
			\item There is a metrizable compactum $A$, a metrizable $σ$-compact space $B$, and a closed set $F ⊆ A × B$ such that $\set{F^b \given b ∈ B} \homeo \C$.
			\item There is a closed set $F ⊆ [0, 1]^ω × 2^ω$ such that $\set{F^b: b ∈ 2^ω} \homeo \C$, or $\C = ∅$.
		\end{enumerate}
	\end{theorem}
	
	\begin{theorem} \labelblock{thm:Polishable_characterization}
		The following conditions are equivalent for a class $\C$ of topological spaces.
		\begin{enumerate}
			\item $\C$ is Polishable. \loclabel{composition}
			\item There is a Polish space $A$ and a closed equivalence relation $E ⊆ A × A$ such that $\set{E^a: a ∈ A} \homeo \C \setminus \set{∅}$. \loclabel{equivalence}
			\item There is a Polish space $A$, an analytic space $B$, and a $G_δ$ set $F ⊆ A × B$ such that $\set{F^b \given b ∈ B} \homeo \C$. \loclabel{rectangular}
			\item There is a $G_δ$ set $F ⊆ [0, 1]^ω × ω^ω$ such that $\set{F^b: b ∈ ω^ω} \homeo \C$, or $\C = ∅$. \loclabel{rigid}
			\item There is a closed set $F ⊆ (0, 1)^ω × ω^ω$ such that $\set{F^b: b ∈ ω^ω} \homeo \C$, or $\C = ∅$. \loclabel{rigid_closed}
		\end{enumerate}
		
		\begin{proof}[Proof of Theorem~\ref{thm:compactifiable_characterization} and \ref{thm:Polishable_characterization}] \hfill
		
			\locimpl{composition}{equivalence}.
				For a composition $\A(q\maps A \to B)$ of $\C$ it is enough to consider the equivalence $E := \set{\tuple{a, a'} ∈ A × A: q(a) = q(a')}$ induced by $q$.
			
			\locimpl{equivalence}{rectangular} is trivial if $∅ ∉ \C$. Otherwise we consider a single-point extension $B ⊇ A$ such that $A$ is clopen in $B$ and use the same $E$. Also see Remark~\ref{thm:emptyset}.
			
			\locimpl{rectangular}{composition}.
				We consider the induced rectangular composition $\A_F(q\maps F \to B)$ (see Construction~\ref{con:rectangular}).
				In the compact case with $B$ compact the proof is finished.
				If $B = ⋃_{n ∈ ω} B_n$ for some compacta $B_n$, then each $F ∩ (A × B_n)$ induces a compact composition of $\set{F^b: b ∈ B_n}$, and $\C$ is equivalent to a countable union of compactifiable classes, which is compactifiable by Observation~\ref{thm:countable_union_of_compositions}.
				In the Polish case, there is a Polish space $B'$ and a continuous surjection $f\maps B' \onto B$.
				Let $\A'$ be the pullback of $\A_F$ along $f$ (Construction~\ref{con:pullback}).
				As in Corollary~\ref{thm:subcomposition}, $\A'$ is a composition of $\set{F^b: b ∈ B} \homeo \C$,
				and it is Polish by Lemma~\ref{thm:analytic_pullback}.
			
			\locimpl{composition}{rigid}, \locref{rigid_closed}.
				Let $\A(q\maps A \to B)$ be a \alt{compact}{Polish} composition of $\C$.
				We may suppose that $B$ is nonempty. Otherwise, $\C$ is empty as well.
				Recall that every nonempty metrizable compactum is a continuous image of the Cantor space $2^ω$ and that every nonempty Polish space is a continuous image of the Baire space $ω^ω$, so we may suppose that $B = {}$\alt{$2^ω$}{$ω^ω$} by Construction~\ref{con:pullback}.
				Recall that every separable metrizable space may be embedded into the Hilbert cube $[0, 1]^ω$, so we may suppose that $A ⊆ [0, 1]^ω$.
				Let $F$ be the graph of $q$.
				By the second part of Construction~\ref{con:rectangular}, $\set{F^b: b ∈ B} \homeo \C$ and $F$ is closed in $A × B$.
				Since $A$ is \alt{compact}{Polish}, $A × B$ and so $F$ is \alt{closed}{$G_δ$} in $[0, 1]^ω$.
				This proves \locref{rigid}.
				The proof of \locref{rigid_closed} is analogous and uses the fact that every Polish space may be embedded into $(0, 1)^ω$ as a closed subspace \cite[4.17]{Kechris}.
			
			The implications \locref{rigid}, \locimpl{rigid_closed}{rectangular} are trivial.
		\end{proof}
	\end{theorem}
	
	\begin{observation} \label{thm:components}
		A class $\C$ of nonempty metrizable continua is compactifiable if and only if there exists a metrizable compactum $A$ whose set of components is equivalent to $\C$.
		
		\begin{proof}
			Let $\A(q\maps A \to B)$ be a compact composition of $\C$.
			By Theorem~\ref{thm:compactifiable_characterization} the indexing space $B$ may be taken zero-dimensional (e.g.\ the Cantor space), and hence the spaces $q\fiber{b}$ are precisely the components of $A$.
			
			On the other hand, let $A$ be a metrizable compactum whose set of components is equivalent to $\C$. Let $q\maps A \to B$ be the quotient map induced by the decomposition of $A$ into its components. Since $A$ is a metrizable compactum, the components are equal to the quasi-components, and hence $B$ is totally separated (i.e. points can be separated by clopen sets), in particular Hausdorff. Therefore, $B$ is a metrizable compactum and $q$ induces the desired compact composition.
		\end{proof}
	\end{observation}
	
	Let us conclude this section with basic observations about (non)existence of compactifiable or Polishable classes.
	
	\begin{remark} \label{thm:emptyset}
		If a class $\C$ is \alt{compactifiable}{Polishable}, then so are the classes $\C \setminus \set{∅}$ and $\C ∪ \set{∅}$. This is because if a map $q\maps A \to B$ induces a compact composition, then the maps $q\maps A \to q\im{A}$ and $q\maps A \to B \topsum \set{∞}$ induces compact compositions as well.
		For Polishable $\C$ the case “$\C ∪ \set{∅}$” is the same, but the case “$\C \setminus \set{∅}$” needs a comment. The map $q\maps A \to q\im{A}$ may not directly induce a Polish composition since $q\im{A}$ may not be $G_δ$ in $B$. Nevertheless, it is analytic, so we use Corollary~\ref{thm:subcomposition}.
		In fact, this gives us the composition $\A_E$ for $E = \set{\tuple{a, a'} ∈ A × A: q(a) = q(a')}$.
	\end{remark}
	
	\begin{observation} \label{thm:countable_union_of_compositions}
		Every countable union of \alt{compactifiable}{Polishable} classes is \alt{compactifiable}{Polishable}.
		
		\begin{proof}
			Let $I$ be a set and for every $i ∈ I$ let $\A_i(q_i\maps A_i \to B_i)$ be a composition of a class $\C_i$.
			We consider the \emph{sum composition} $\A(q\maps A \to B) := \TopSum_{i ∈ I} \A_i$, i.e. $A := \TopSum_{i ∈ I} A_i$, $B := \TopSum_{i ∈ I} B_i$, and $q := \TopSum_{i ∈ I} q_i\maps A \to B$.
			Clearly, $\A$ is a composition of $⋃_{i ∈ I} \C$.
			If $I$ is \alt{finite}{countable} and the compositions $\A_i$ are \alt{compact}{Polish}, then $\A$ is also \alt{compact}{Polish}.
			
			It remains to consider a countable sum of compact compositions that is not compact.
			Without loss of generality, $∅ ∉ \C_i ≠ ∅$ for every $i ∈ I$ (Remark~\ref{thm:emptyset}), and so $A$ and $B$ are separable metrizable locally compact non-compact spaces.
			We consider their one-point compactifications $A^+$ and $B^+$, which are metrizable, and the corresponding extension $q^+\maps A^+ \to B^+$ of the map $q$.
			The map $q^+$ is continuous since $q$ is perfect (i.e. closed with compact fibers), and it induces a composition of $⋃_{i ∈ I} \C_i ∪ \set{\set{∞}}$, so if the given classes contain a one-point space, we are done.
			Otherwise, we take any space $C ∈ ⋃_{i ∈ I} \C_i$, attach it to the point $∞ ∈ A^+$, and modify the definition of $q^+$ accordingly.
		\end{proof}
	\end{observation}
	
	\begin{corollary}
		Every countable family of metrizable compacta is compactifiable. Every countable family of Polish spaces is Polishable.
	\end{corollary}
	
	\begin{remark}
		We require metrizability (or equivalently existence of a countable base) in the definition of compact composition not only to obtain a notion stronger than Polish composition, but because otherwise the corresponding compactifiability would be trivial.
		Using the one-point compactification as in the previous proof, we may easily construct a composition with compact composition space and compact indexing space for any family of compacta.
	\end{remark}
	
	\begin{observation}
		By Theorem~\ref{thm:Polishable_characterization} there are at most $\continuum$-many nonequivalent Polishable classes since there are only $\continuum$-many $G_δ$ subsets of $[0, 1]^ω × ω^ω$.
		On the other hand, there are $\continuum$-many non-homeomorphic metrizable compact spaces – even in the real line. Hence, there are exactly $2^\continuum$-many nonequivalent classes of metrizable compacta and also exactly $2^\continuum$-many nonequivalent classes of Polish spaces. This cardinal argument gives us that many classes of metrizable compacta are not Polishable.
	\end{observation}

%% file: hyperspaces.tex
\section{Compactifiability and hyperspaces} \label{sec:hyperspaces}
	
	A class of topological spaces is often equivalent to a family of subspaces of some fixed ambient space. Therefore, it is natural to consider how compactifiability of such family is related to its properties when viewed as a subset of a hyperspace.
	
	For a topological space $X$ we shall consider the hyperspaces of all subsets $\powset{X}$, of all closed subsets $\Closed{X}$, of all compact subsets $\Compacta{X}$, and of all subcontinua $\Continua{X}$ endowed with the \emph{Vietoris topology}.
	We include the empty set in the families.
	Recall that the \emph{lower Vietoris topology} $τ_V^-$ is generated by the sets $U^- = \set{A: A ∩ U ≠ ∅}$ for $U ⊆ X$ open, and the \emph{upper Vietoris topology} $τ_V^+$ is generated by the sets $U^+ = \set{A: A ⊆ U}$ for $U ⊆ X$ open. The Vietoris topology $τ_V$ is their join.
	
	Also recall that if $X$ is metrizable by a metric $d$, the corresponding \emph{Hausdorff metric} $d_H$ on $\Closed{X}$ is defined by $d_H(A, B) = \max(δ(A, B), δ(B, A))$ where $δ(A, B) = \sup_{x ∈ A} d(x, B) = \sup_{x ∈ A} \inf_{y ∈ B} d(x, y) = \inf\set{ε: A ⊆ N_ε(B)}$.
	We have $δ(∅, B) = 0$ for every $B$, and $δ(A, ∅) = ∞$ for every $A ≠ ∅$, and also $δ(A, B) = ∞$ for every $A$ unbounded and $B$ bounded. Hence, strictly speaking, $d_H$ is an extended metric, but we may always cap it at $1$ or suppose that $d ≤ 1$ and interpret the infima in $[0, 1]$, so $\inf ∅ = 1$. In any case, the singleton $\set{∅}$ is clopen in $\Closed{X}$ with both Vietoris topology and Hausdorff metric topology.
	
	The Vietoris topology and the topology induced by the Hausdorff metric are not comparable on $\Closed{X}$ in general, but they coincide on $\Compacta{X}$. If $X$ is compact or Polish, so is $\Compacta{X}$. Also, $\Continua{X}$ is a closed subspace of $\Compacta{X}$ if $X$ is Hausdorff. For reference on the mentioned properties see \cite[4.F]{Kechris}.
	
	\begin{construction}[from hyperspace to composition] \label{con:from_hyperspace}
		Let $X$ be a topological space and let $\F ⊆ \powset{X}$. We consider the set $A_\F := \set{\tuple{x, F}: x ∈ F ∈ \F} ⊆ X × \F$. Let us denote the corresponding composition (Construction~\ref{con:rectangular}) by $\A_\F$. Since $(A_\F)^F = F$ for every $F ∈ \F$, we have that $\A_\F$ is a composition of the family $\F$ with composition space $A_\F$ and indexing space $\F$.
		The composition map is just the projection $π_\F\restr{A_\F}$.
		Also, $A_\F = \R_∈ ∩ (X × \F)$ where $\R_∈ := \set{\tuple{x, F} ∈ X × \powset{X}: x ∈ F}$ is the membership relation.
	\end{construction}
	
	\begin{observation} \label{thm:membership_closed}
		If $X$ is a regular space, then the membership relation of closed sets is closed, i.e.\ $\R_∈ ∩ (X × \Closed{X})$ is closed in $X × \Closed{X}$ (even with respect to $τ_V^+$).
		
		\begin{proof}
			If $F ∈ \Closed{X}$ and $x ∈ X \setminus F$, then there are disjoint open sets $U, V ⊆ X$ such that $x ∈ U$ and $F ⊆ V$. We have that $U × V^+$ is a neighborhood of $\tuple{x, F}$ disjoint with $\R_∈$.
		\end{proof}
	\end{observation}
	
	\begin{proposition} \label{thm:family}
		\hfill
		\begin{enumerate}
			\item If $X$ is a metrizable compactum and $\F$ is an $F_σ$ subset of \alt{$\Compacta{X}$}{$\Continua{X}$}, then $\F$ is a compactifiable class of \alt{compacta}{continua}.
			\item If $X$ is a Polish space and $\F$ is an analytic subset of \alt{$\Compacta{X}$}{$\Continua{X}$}, then $\F$ is a Polishable class of \alt{compacta}{continua}.
		\end{enumerate}
		
		\begin{proof}
			It is enough to use the set $A_\F ⊆ X × \F$ from Construction~\ref{con:from_hyperspace} and Theorem~\ref{thm:compactifiable_characterization} and \ref{thm:Polishable_characterization}.
		\end{proof}
	\end{proposition}
	
	Next, we shall introduce a construction in the opposite direction, i.e.\ turning a composition into a subset of a hyperspace. But first, let us recall some further properties of hyperspaces and their induced maps.
	
	\begin{observation} \label{thm:hyperspace}
		If a space $X$ is identified with the family of its singletons $\subsets{1}{X}$, then it becomes a subspace of $\powset{X}$ with respect to all $τ_V^-$, $τ_V^+$, and $τ_V$ since for every open $U ⊆ X$ we have $U^- ∩ \subsets{1}{X} = U^+ ∩ \subsets{1}{X} = \subsets{1}{U}$.
	\end{observation}
	
	\begin{notation}
		Let $f\maps X \to Y$ be a map between sets. We shall use the notation for induced maps from \cite[5.9]{Michael}:
		\begin{itemize}
			\item $\ImageMap{f}\maps \powset{X} \to \powset{Y}$ is the \emph{image map} defined by $\ImageMap{f}(A) = f\im{A}$,
			\item $\FiberMap{f}\maps Y \to \powset{X}$ is the \emph{fiber map} defined by $\FiberMap{f}(y) = f\fiber{y}$,
			\item $\PreimageMap{f}\maps \powset{Y} \to \powset{X}$ is the \emph{preimage map} defined by $\PreimageMap{f}(B) = f\preim{B}$.
		\end{itemize}
	\end{notation}
	
	The following proposition summarizes properties of the induced maps defined above.
	Some of the equivalences were proved by Michael~\cite[5.10]{Michael}.
	Note that our map $f$ does not have to be onto, we include the empty set in the hyperspace, and we also formulate the equivalences separately for $τ_V^-$ and $τ_V^+$.
	
	\begin{proposition} \label{thm:induced_maps}
		Let $f\maps X \to Y$ be a map between topological spaces.
		\begin{enumerate}
			\item $f$ is continuous $\iff$ $\ImageMap{f}$ is $τ_V^-$-cont.\ $\iff$ $\ImageMap{f}$ is $τ_V^+$-cont.\ $\iff$ $\ImageMap{f}$ is $τ_V$-cont.
			\item $f$ is an embedding $\iff$ $\ImageMap{f}$ is $τ_V^-$-emb.\ $\iff$ $\ImageMap{f}$ is $τ_V^+$-emb.\ $\iff$ $\ImageMap{f}$ is $τ_V$-emb.
			\item $f$ is an open embedding $\iff$ $\ImageMap{f}$ is $τ_V^+$-open emb.\ $\iff$ $\ImageMap{f}$ is $τ_V$-open emb.
			\item $f$ is a closed embedding $\iff$ $\ImageMap{f}$ is $τ_V^-$-closed emb.\ $\iff$ $\ImageMap{f}$ is $τ_V$-closed emb.
			\item $f$ is open $\iff$ $\FiberMap{f}$ is $τ_V^-$-continuous $\iff$ $\PreimageMap{f}$ is $τ_V^-$-continuous.
			\item $f$ is closed $\iff$ $\FiberMap{f}$ is $τ_V^+$-continuous $\iff$ $\PreimageMap{f}$ is $τ_V^+$-continuous.
			\item $f$ is closed and open $\iff$ $\FiberMap{f}$ is $τ_V$-continuous $\iff$ $\PreimageMap{f}$ is $τ_V$-continuous.
			\item $f$ is continuous $\implies$ $\FiberMap{f}$ is $τ_V$-(closed and open) onto its image.
		\end{enumerate}
		
		\begin{proof}[Proof sketch]
			We use the following equalities.
			\begin{gather*}
				\begin{aligned}
					(\ImageMap{f})\preim{B^-} &= f\preim{B}^- & (\ImageMap{f})\preim{B^+} &= f\preim{B}^+ \\
					(\ImageMap{f})\im{A^-} &= f\im{A}^- ∩ \rng(\ImageMap{f}) & (\ImageMap{f})\im{A^+} &= f\im{A}^+ ⊆ \rng(\ImageMap{f})
						\quad \text{for $f$ injective} \\
					(\FiberMap{f})\preim{A^-} &= f\im{A} & (\FiberMap{f})\preim{A^+} &= f^∀\im{A} := Y \setminus f\im{X \setminus A} \\
					(\PreimageMap{f})\preim{A^-} &= f\im{A}^- & (\PreimageMap{f})\preim{A^+} &= f^∀\im{A}^+ \\
				\end{aligned} \\
				(\FiberMap{f})\im{B} = \begin{cases}
					f\preim{B}^- ∩ \rng(\FiberMap{f}) \quad \text{if $B ⊆ \rng(f)$} \\
					f\preim{B}^+ ∩ \rng(\FiberMap{f}) \quad \text{if $B ⊈ \rng(f)$ or $f$ is onto } \\
				\end{cases}
			\end{gather*}
			Regarding the embeddings, if $f$ is an embedding and $U ⊆ X$ is open, then $f\im{U} = V ∩ \rng(f)$ for some open $V ⊆ Y$. Therefore, we have 
			\begin{align*}
				\ImageMap{f}\im{U^-} &= f\im{U}^- ∩ \rng(f) = (V ∩ \rng(f))^- ∩ \rng(\ImageMap{f}) = V^- ∩ \rng(\ImageMap{f}), \\
				\ImageMap{f}\im{U^+} &= f\im{U}^+ = (V ∩ \rng(f))^+ = V^+ ∩ \rng(\ImageMap{f}),
			\end{align*}
			and so $\ImageMap{f}$ is a $τ_V^-$-, $τ_V^+$- and hence a $τ_V$-embedding. 
			Regarding the closedness and openness, observe that $\rng(\ImageMap{f}) = \rng(f)^+$, so if $\rng(f)$ is open, then $\rng(\ImageMap{f})$ is $τ_V^+$-open, and if $\rng(f)$ is closed, then $\rng(\ImageMap{f})$ is $τ_V^-$-closed.
			For the backward implications we may use Observation~\ref{thm:hyperspace} since $f$ may be viewed as a restriction $\subsets{1}{X} \to \subsets{1}{Y}$ of $\ImageMap{f}$, and $\rng(f)$ is essentially $\rng(\ImageMap{f}) ∩ \subsets{1}{Y}$.
		\end{proof}
	\end{proposition}
	
	\begin{definition}
		A composition $\A(q\maps A \to B)$ is called a \emph{strong composition} if the composition map $q$ is closed and open and $\card{B \setminus \rng(q)} ≤ 1$.
		A class $\C$ of topological spaces is called \alt{\emph{strongly compactifiable}}{\emph{strongly Polishable}} if there is a \alt{strong compact}{strong Polish} composition of $\C$.
		
		The strongness of a composition means that the corresponding decomposition of $A$ is continuous (closedness correspond to upper semi-continuity and openness to lower semi-continuity).
		Note that the rather technical condition $\card{B \setminus \rng(q)} ≤ 1$ and also clopenness of $\rng(q)$ can be obtained for every composition by removing $B \setminus \rng(q)$ and then eventually adding a clopen point (Remark~\ref{thm:emptyset}). Also, the closedness of $q$ is trivial for compact compositions.
	\end{definition}
	
	\begin{construction}[from composition to hyperspace] \label{con:to_hyperspace}
		To every composition $\A(q\maps A \to B)$ we assign the disjoint family $\F_\A := \set{q\fiber{b}: b ∈ B} ⊆ \powset{A}$.
		
		We have $\FiberMap{q}\maps B \onto \F_\A ⊆ \powset{A}$, so we have two natural topologies on $\F_\A$ – the quotient topology induced by $\FiberMap{q}$ from $B$, and the subspace topology induced from the hyperspace $\powset{A}$.
		By Proposition~\ref{thm:induced_maps} the Vietoris topology is finer than the quotient topology. The converse holds if and only if $q$ is both closed and open.
		The map $\FiberMap{q}$ is a homeomorphism with respect to the quotient topology if and only if it is a bijection, which happens if and only if $\card{B \setminus \rng(q)} ≤ 1$.
		Therefore, $\F_\A$ is homeomorphic to $B$ via $\FiberMap{q}$ if and only if the composition $\A$ is strong.
		
		In this case, if $\A$ is a \alt{compact}{Polish} composition of compacta, then $\F_\A$ is \alt{compact}{Polish}, and so it is a \alt{closed}{$G_δ$} subset of the \alt{compact}{Polish} hyperspace $\Compacta{A}$.
	\end{construction}
	
	\begin{observation} \label{thm:closed_family}
		If $\A(q\maps A \to B)$ is a strong composition, then the family $\F_\A$ is closed in every Hausdorff space $\H ⊆ \powset{A}$ containing it.
		
		\begin{proof}
			Let us consider the family $\unions{\F} := \set{F ∈ \H: q\preim{q\im{F}} = F}$, which is closed since $\PreimageMap{q} ∘ \ImageMap{q}$ is continuous and $\H$ is Hausdorff, and the family $\downset{\F} := (\ImageMap{q})\preim{\subsets{≤1}{B}}$ (where $\subsets{≤1}{B}$ denotes the family of all subsets of $B$ with at most one element), which is also closed since $B \homeo \F_\A$ is Hausdorff, and so $\subsets{≤1}{B}$ is closed in $\powset{B}$. To conclude, it is enough to observe that $\F_\A ⊆ \unions{\F} ∩ \downset{\F} ⊆ \F_\A ∪ \set{∅}$.
		\end{proof}
	\end{observation}
	
	\begin{lemma} \label{thm:clopen_projection}
		\def \rnb{r.n.b.} 
		
		Let $X$, $Y$ be topological spaces, and let $R ⊆ X × Y$. Let us consider the map $ρ\maps Y \to \powset{X}$ defined by $ρ(y) := R^y$.
		\begin{enumerate}
			\item The map $π_Y\restr{R}\maps R \to Y$ is open if and only if the map $ρ$ is $τ_V^-$-continuous.
			\item The map $π_Y\restr{R}\maps R \to Y$ is closed if and only if the map $ρ$ is $τ_V^+$-continuous and every set $R^y × \set{y}$ has a rectangular neighborhood basis (\rnb{}), i.e. every its neighborhood in $R$ contains a neighborhood of form $R ∩ (U × V)$ for some open sets $U$ and $V$. The \rnb{} condition is satisfied if $\rng(ρ) ⊆ \Compacta{X}$.
		\end{enumerate}
		
		\begin{proof}
			The necessity of $τ_V^{+/-}$-continuity follows from equality $ρ = \ImageMap{π_X} ∘ \FiberMap{(π_Y\restr{R})}$ and from Proposition~\ref{thm:induced_maps}.
			The open case follows from equality $π_Y\im{R ∩ (U × V)} = \set{y ∈ V: R^y ∩ U ≠ ∅} = ρ\preim{U^-} ∩ V$.
			The map $π_Y\restr{R}$ is closed if and only if for every closed $F ⊆ R$ and every $y ∈ Y \setminus π_Y\im{F}$ there is an open neighborhood $W$ of $y$ disjoint with $π_Y\im{F}$. Considering $R ∩ (X × W)$ gives us necessity of the \rnb{} condition.
			On the other hand, if $U × V$ is an open neighborhood of $R^y × \set{y} ≠ ∅$ disjoint with $F$, then we put $W := ρ\preim{U^+} ∩ V$. Note that $z ∈ ρ\preim{U^+}$ if and only if $R^z ⊆ U$. Hence, if $\tuple{x, z} ∈ R$ and $z ∈ W$, then $\tuple{x, z} ∈ U × V$ and so it cannot be in $F$. If $R^y = ∅$, then we put $W := Y \setminus π_Y\im{R}$, which is open since $π_Y\im{R} = ρ\preim{X^-}$ and $X^-$ is $τ_V^+$-closed.
			The \rnb{} condition holds if every $R^y × \set{y}$ is compact by the tube lemma \cite[3.1.15]{Engelking}.
		\end{proof}
	\end{lemma}
	
	\begin{corollary} \label{thm:hyperspace_compositions_are_strong}
		Let $\A_\F(q\maps A_\F \to \F)$ be the composition obtained by Construction~\ref{con:from_hyperspace} from a family $\F ⊆ \powset{X}$.
		We have that the map $q$ is open and $\card{\F \setminus \rng(q)} ≤ 1$.
		If $\F ⊆ \Compacta{X}$, then $q$ is also closed, and hence the composition is strong.
		
		\begin{proof}
			The map $q$ is the projection $\R_∈ ∩ (X × \F) \to \F$, so we may use Lemma~\ref{thm:clopen_projection}.
			The corresponding map $ρ$ is $\id\maps \F \to \powset{X}$, which is both $τ_V^-$- and $τ_V^+$-continuous.
			The fact that $\card{\F \setminus \rng(q)} ≤ 1$ is clear since there is only one empty set.
		\end{proof}
	\end{corollary}
	
	\begin{corollary} \label{thm:strong_pullback}
		Let $\A(q\maps A \to B)$ be a composition of spaces $\tuple{A_b}_{b ∈ B}$, let $f\maps B' \to B$ be a continuous map, and let $\A'(q'\maps A' \to B')$ be the pullback of $\A$ along $f$ (Construction~\ref{con:pullback}).
		If $q$ is open, so is $q'$.
		If $q$ is closed and every space $A_b$ is compact, then $q'$ is also closed.
		It follows that strong compositions of compact spaces are preserved by pullbacks (such that $\card{f\preim{B \setminus \rng(f)}} ≤ 1$).
		
		\begin{proof}
			We apply Lemma~\ref{thm:clopen_projection} to $A' ⊆ A × B'$.
			The corresponding map $ρ$ is $\FiberMap{q} ∘ f$, which is \alt{$τ_V^-$-}{$τ_V^+$-}continuous if $q$ is \alt{open}{closed} by Proposition~\ref{thm:induced_maps}.
		\end{proof}
	\end{corollary}
	
	By putting all the previous claims and propositions together, we obtain the following characterizations – compare with Theorem~\ref{thm:compactifiable_characterization} and \ref{thm:Polishable_characterization}.
	
	\begin{theorem} \label{thm:strongly_compactifiable_characterization}
		The following conditions are equivalent for a class $\C$ of topological spaces.
		\begin{enumerate}
			\item $\C$ is strongly compactifiable.
			\item There is a metrizable compactum $X$ and a closed family $\F ⊆ \Compacta{X}$ such that $\F \homeo \C$.
			\item There is a closed zero-dimensional disjoint family $\F ⊆ \Compacta{[0, 1]^ω}$ such that $\F \homeo \C$.
		\end{enumerate}
	\end{theorem}
	
	\begin{theorem} \label{thm:strongly_Polishable_characterization}
		The following conditions are equivalent for a class $\C$ of topological spaces.
		\begin{enumerate}
			\item $\C$ is a strongly Polishable class of compacta.
			\item There is a Polish space $X$ and an analytic family $\F ⊆ \Compacta{X}$ such that $\F \homeo \C$.
			\item There is a $G_δ$ zero-dimensional disjoint family $\F ⊆ \Compacta{[0, 1]^ω}$ such that $\F \homeo \C$.
			\item There is a closed zero-dimensional disjoint family $\F ⊆ \Compacta{(0, 1)^ω}$ such that $\F \homeo \C$.
		\end{enumerate}
	\end{theorem}
	
	\begin{proof}
		Let $\F ⊆ \Compacta{X}$.
		Construction~\ref{con:from_hyperspace} gives us the corresponding composition $\A_\F$, which is strong by Corollary~\ref{thm:hyperspace_compositions_are_strong}.
		If $X$ is a metrizable compactum and $\F$ is closed, then the composition $\A_\F$ is compact.
		If $X$ is Polish and $\F$ is analytic, then there is a continuous surjection $f\maps Y \onto \F$ from a Polish space $Y$ such that $\card{f\fiber{∅}} ≤ 1$.
		The pullback of $\A_\F$ along $f$ (Construction~\ref{con:pullback}) is a composition of $\F$ that is Polish by Lemma~\ref{thm:analytic_pullback} and strong by Corollary~\ref{thm:strong_pullback}.
		
		On the other hand, let $\A(q\maps A \to B)$ be a \alt{strong compact}{strong Polish} composition of $\C$.
		Without loss of generality, $B$ is zero-dimensional (we use Construction~\ref{con:pullback} as in Theorem~\ref{thm:compactifiable_characterization} and \ref{thm:Polishable_characterization} together with Corollary~\ref{thm:strong_pullback}).
		Construction~\ref{con:to_hyperspace} gives us the corresponding zero-dimensional disjoint family $\F_\A ⊆ \Compacta{A}$, which is closed by Observation~\ref{thm:closed_family}.
		There is an embedding $e\maps A \into [0, 1]^ω$, and so $\ImageMap{e}\maps\Compacta{A} \into \Compacta{[0, 1]^ω}$ is an embedding by Proposition~\ref{thm:induced_maps}.
		In the compact case, $\ImageMap{e}\im{\F_\A}$ is compact and so closed in $\Compacta{[0, 1]^ω}$.
		In the Polish case, $\ImageMap{e}\im{\F_\A}$ is Polish and so $G_δ$ in $\Compacta{[0, 1]^ω}$.
		Moreover, there is a closed embedding $i\maps A \into (0, 1)^ω$ by \cite[4.17]{Kechris}, and so $\ImageMap{i}\maps \Compacta{A} \into \Compacta{(0, 1)^ω}$ is a closed embedding by Proposition~\ref{thm:induced_maps}.
		Hence, $\ImageMap{i}\im{\F_\A}$ is a closed subset of $\Compacta{(0, 1)^ω}$.
		
		The remaining implications are trivial.
	\end{proof}
	
	\begin{lemma} \label{thm:semicontinuous_distance}
		Let $X$ be a metric space and let $\R$ denote the family $\set{\tuple{A, B} ∈ \Compacta{X}^2: A ⊆ B}$ viewed as a subspace of $\tuple{\Compacta{X}, τ_V} × \tuple{\Compacta{X}, τ_V^+}$.
		The Hausdorff metric $d_H\maps \R \to [0, ∞)$ is upper semi-continuous.
		
		\begin{proof}
			Let $\tuple{A, B} ∈ \R$ and $r > d_H(A, B)$.
			We want to find $\U$ a $τ_V$-neighborhood of $A$ and $\V$ a $τ_V^+$-neighborhood of $B$ such that $d_H(A', B') < r$ for every $A' ∈ \U$ and $B' ∈ \V$.
			
			For every $\tuple{A', B'} ∈ \R$ we have $d_H(A', B') = δ(B', A') = \inf\set{ε > 0: B' ⊆ N_ε(A')}$.
			Hence, $d_H(A', B') = δ(B', A') ≤ δ(B', B) + δ(B, A) + δ(A, A') ≤ δ(B', B) + d_H(A, B) + d_H(A, A')$.
			
			Let $ε > 0$ such that $d_H(A, B) + 2ε < r$.
			We put $\U := \set{A': d_H(A, A') < ε}$ and $\V := N_ε(B)^+$.
			The set $\U$ is $τ_V$-open since the Hausdorff metric topology coincides with the Vietoris topology on $\Compacta{X}$, and $\V$ is clearly $τ_V^+$-open.
			Moreover, for every $B' ∈ \V$ we have $δ(B', B) ≤ ε$.
			Therefore, for every $\tuple{A', B'} ∈ \U × \V$ we have $d_H(A', B') < d_H(A, B) + 2ε < r$.
		\end{proof}
	\end{lemma}
	
	\begin{proposition} \label{thm:compactifiable_G_delta}
		Let $\A(q\maps A \to B)$ be a Polish composition of compacta such that the composition map $q$ is closed.
		The family $\F_\A ⊆ \Compacta{A}$ obtained via Construction~\ref{con:to_hyperspace} is $G_δ$.
		
		\begin{proof}
			As in the proof of Observation~\ref{thm:closed_family} we have $\F_\A ⊆ \unions{\F} ∩ \downset{\F} ⊆ \F_\A ∪ \set{∅}$, and the family $\downset{\F}$ is closed.
			But the family $\unions{\F} = \set{F ∈ \Compacta{A}: \widehat{F} := q\preim{q\im{F}} = F}$ is now not necessarily closed since the map $\PreimageMap{q} ∘ \ImageMap{q}$ is not necessarily continuous.
			It is only $τ_V^+$-continuous since $q$ is closed.
			
			Let $d$ be a compatible metric on $A$ and let $\G_n := \set{F ∈ \Compacta{A}: d_H(F, \widehat{F}) < \frac{1}{n}}$. Clearly, $\unions{\F} = ⋂_{n ∈ ℕ} \G_n$, so it is enough to show that each $\G_n$ is open.
			Let $\R$ be the space from Lemma~\ref{thm:semicontinuous_distance} for the base space $A$.
			The map $\id_{\F_\A} \diag (\PreimageMap{q} ∘ \ImageMap{q})\maps \F_\A \to \R$ that maps $F \mapsto \tuple{F, \widehat{F}}$ is continuous since $\PreimageMap{q} ∘ \ImageMap{q}$ is $τ_V^+$-continuous.
			By Lemma~\ref{thm:semicontinuous_distance} the map $d_H\maps \R \to [0, ∞)$ is upper semi-continuous.
			Together, the map $F \mapsto d_H(F, \widehat{F})$ is upper semi-continuous, and the families $\G_n$ are open.
		\end{proof}
	\end{proposition}
	
	\begin{corollary}
		Every compactifiable class is strongly Polishable.
		Also, in the definition of strong Polishability it is enough that the witnessing composition map is closed.
	\end{corollary}

	We have shown that every compactifiable class is strongly Polishable.
	On the other hand, strongly Polishable classes of compacta are sometimes close to being compactifiable.
	Compare the following characterization with Theorem~\ref{thm:compactifiable_characterization} and \ref{thm:Polishable_characterization}.
	
	\begin{theorem} \labelblock{thm:strongly_Polishable_rectangular_characterization}
		The following conditions are equivalent for a class $\C$ of topological spaces.
		\begin{enumerate}
			\item $\C$ is a strongly Polishable class of compacta. \loclabel{composition}
			\item There is a metrizable compactum $A$, an analytic space $B$, and a closed set $F ⊆ A × B$ such that $\set{F^b: b ∈ B} \homeo \C$. \loclabel{rectangular}
			\item There is a closed set $F ⊆ [0, 1]^ω × ω^ω$ such that $\set{F^b: b ∈ ω^ω} \homeo \C$, or $\C = ∅$. \loclabel{rigid}
			\item There is a closed set $F ⊆ [0, 1]^ω × 2^ω$ and a $G_δ$ set $G ⊆ 2^ω$ such that $\set{F^b: b ∈ G} \homeo \C$ and $\set{F^b: b ∈ 2^ω} = \clo{\set{F^b: b ∈ G}}$ in $\Compacta{[0, 1]^ω}$, or $\C = ∅$. \loclabel{rigid_compact}
		\end{enumerate}
		
		\begin{proof}
			\locimpl{rectangular}{composition}: Let $f\maps B' \onto B$ be a continuous surjection from a Polish space $B'$.
			Let $F'$ denote the set $\set{\tuple{a, b'} ∈ A × B': \tuple{a, f(b')} ∈ F}$, which is closed as a continuous preimage of $F$.
			The induced rectangular composition $\A_{F'}(q\maps F' \to B')$ is a Polish composition of $\C$ (cf. Lemma~\ref{thm:analytic_pullback}).
			The map $q = π_{B'}\restr{F'}$ is closed by the Kuratowski theorem \cite[Theorem~3.1.16]{Engelking} since $A$ is compact.
			Therefore, $\C$ is strongly Polishable by Proposition~\ref{thm:compactifiable_G_delta}.
			
			\locimpl{composition}{rigid} and \locimpl{composition}{rigid_compact}:
			By Theorem~\ref{thm:strongly_Polishable_characterization} there is a $G_δ$ family $\F ⊆ \Compacta{[0, 1]^ω}$ equivalent to $\C$.
			We may suppose that $\F$ is nonempty.
			For \locref{rigid} we consider the composition $\A_\F$ (Construction~\ref{con:from_hyperspace}) and its pullback (Construction~\ref{con:pullback}) along a continuous surjection $f\maps ω^ω \onto \F$, i.e. $F := \set{\tuple{x, y} ∈ [0, 1]^ω × ω^ω: x ∈ f(y)}$.
			For \locref{rigid_compact} we do the same, but with $\clo{\F}$ and a continuous surjection $f\maps 2^ω \onto \clo{\F}$, i.e. $F := \set{\tuple{x, y} ∈ [0, 1]^ω × 2^ω: x ∈ f(y)}$, and we put $G := f\preim{\F}$.
			Clearly, we have $\set{F^b: b ∈ G} = \F$ and $\set{F^b: b ∈ 2^ω} = \clo{\F}$.
			
			The remaining implications are trivial.
		\end{proof}
	\end{theorem}
	
	The last condition condition of the previous theorem is quite close to compactifiability.
	It is enough to modify the fibers $F^b$ for $b ∈ 2^ω \setminus G$, so they become spaces from the given class $\C$, while keeping the modified set $F' ⊆ [0, 1]^ω × 2^ω$ closed.
	
	\begin{theorem} \label{thm:fiber_fixing}
		Let $\tuple{X, d}$ be a metric compactum and for every $n ∈ ω$ let $\A_n$ be a finite covering of $X$ by closed sets of diameter $< 2^{-n}$.
		For every $F ∈ \Compacta{X}$ let $\A_n(F)$ denote the space $⋃\set{A ∈ \A_n: A ∩ F ≠ ∅}$.
		Every $G_δ$ family $\F ⊆ \Compacta{X}$ containing a copy of every space from $\set{\A_n(F): F ∈ \clo{\F}, n ∈ ω}$ is compactifiable.
		
		\begin{proof}
			If $\F = ∅$, the theorem holds.
			Otherwise, there is a continuous surjection $f\maps 2^ω \onto \clo{\F}$.
			The set $G := f\preim{\F}$ is $G_δ$ in $2^ω$, and so its complement can be written as a disjoint union $⋃_{n ∈ ω} K_n$ of compact sets.
			As before, we consider the pullback of the induced composition of $\clo{\F}$, i.e. the closed set $F := \set{\tuple{x, b} ∈ X × 2^ω: x ∈ f(b)}$.
			For every $n ∈ ω$ let $H_n := ⋃\set{\A_n(F^b) × \set{b}: b ∈ K_n} ⊆ X × K_n$, and let $F' := F ∪ ⋃_{n ∈ ω} H_n$.
			
			We need to prove that $F'$ is closed.
			Then it is clear that $F'$ induces a compact composition of $\F$ since $\set{(F')^b: b ∈ G} = \F$ and $(F')^b = \A_n(F^b)$ for $b ∈ K_n$.
			Every set $H_n$ is closed since it is equal to $⋃_{A ∈ \A_n} A × (f\preim{A^- ∩ \clo{\F}} ∩ K_n)$.
			Moreover, we have $H_n ⊆ N_{2^{-n}}(F)$ for a suitable metric on $X × 2^ω$.
			Altogether, $\clo{F'} = F ∪ ⋃_{n ∈ ω} H_n ∪ ⋂_{k ∈ ω} \clo{⋃_{n ≥ k} H_n}$, and the last term is below $⋂_{k ∈ ω} \clo{N_{2^{-k}}(F)} = F$.
		\end{proof}
	\end{theorem}
	
	\begin{lemma} \label{thm:G_delta}
		Let $X$ be a Polish space such that $X × ω^ω$ embeds into $X$.
		Every analytic family $\F ⊆ \Compacta{X}$ is equivalent to a $G_δ$ family $\G ⊆ \Compacta{X}$.
		
		\begin{proof}
			There is a Polish space $Y ∈ \set{∅, ω^ω, ω^ω \topsum 1}$ and a continuous surjection $f\maps Y \onto \F$ such that $\card{f\fiber{∅}} ≤ 1$.
			As in the proof of Theorem~\ref{thm:strongly_Polishable_characterization}, the pullback $\A'(q\maps A' \to Y)$ of the composition $\A_\F$ along $f$ is a strong Polish composition of $\F$.
			Hence, the corresponding family of fibers $\F_{\A'} ⊆ \Compacta{A'}$ is $G_δ$ (Construction~\ref{con:to_hyperspace}).
			Since the composition space $A'$ is a subspace of $X × Y$, it embeds into $X × ω^ω$, and so into $X$.
			For an embedding $e\maps A' \into X$, the induced map $\ImageMap{e}\maps \Compacta{A'} \into \Compacta{X}$ is also an embedding by Proposition~\ref{thm:induced_maps}.
			Since $A'$ and $\Compacta{A'}$ are Polish spaces, $\ImageMap{e}\im{\F_{\A'}} ⊆ \Compacta{X}$ is the desired $G_δ$ family equivalent to $\F$.
		\end{proof}
	\end{lemma}
	
	\begin{corollary} \labelblock{thm:fiber_fixing_application}
		We have the following applications of the previous theorem.
		\begin{enumerate}
			\item Every analytic subset of $\Continua{[0, 1]^ω}$ containing a copy of every Peano continuum is compactifiable.
				In particular, the class of all Peano continua is compactifiable. \loclabel{Peano}
			\item Every analytic subset of $\Compacta{2^ω}$ containing a copy of $2^ω$ is compactifiable. \loclabel{Cantor}
			\item Every $G_δ$ subset of $\Continua{D_ω}$ containing a copy of $D_ω$ is compactifiable ($D_ω$ denotes the Ważewski's universal dendrite \cite[10.37]{Nadler}). \loclabel{Wazewski}
		\end{enumerate}
		
		\begin{proof}
			Let $X$ denote $[0, 1]^ω$ or $2^ω$ or $D_ω$, respectively.
			By Lemma~\ref{thm:G_delta} there is a $G_δ$ family $\F ⊆ \Compacta{X}$ equivalent to the original family.
			By Theorem~\ref{thm:fiber_fixing} it is enough to find suitable coverings $\A_n$ of $X$ such that every space from $\G := \set{\A_n(F): F ∈ \clo{\F}, n ∈ ω}$ is homeomorphic to a space from $\F$.
			We cover $X$ by its copies of sufficiently small diameters.
			In \locref{Peano} every space from $\G$ is a connected finite union Hilbert cubes, and so a Peano continuum.
			In \locref{Cantor} every space from $\G$ is a finite union of Cantor spaces, and so a Cantor space.
			In \locref{Wazewski} every space from $\G$ is a connected finite union of copies of $D_ω$ in $D_ω$, and so a copy of $D_ω$ if we choose the coverings so that for every $A ∈ ⋃_{n ∈ ω} \A_n$ all branching points of $D_ω$ in $A$ are in the interior of $A$.
		\end{proof}
	\end{corollary}
	
	In Theorem~\ref{thm:strongly_Polishable_rectangular_characterization} we have characterized strong Polishability in the language of rectangular compositions to make a connection with compactifiability.
	Now we characterize compactifiability using families in hyperspaces to make a connection with strong compactifiability.
	
	\begin{theorem} \labelblock{thm:compactifiable_hyperspace_characterization}
		The following conditions are equivalent for a class $\C$ of topological spaces.
		\begin{enumerate}
			\item $\C$ is compactifiable. \loclabel{compactifiable}
			\item There is a metrizable compactum $X$ and a family $\F ⊆ \Compacta{X}$ such that $\F \homeo \C$ and $\tuple{\F, τ}$ is a metrizable compactum for a topology $τ ⊇ τ_V^+$. \loclabel{weak_family}
			\item There is a $G_δ$ disjoint family $\F ⊆ \Compacta{[0, 1]^ω}$ such that $\F \homeo \C$ and $\tuple{\F, τ_V^+}$ is a zero-dimensional metrizable compactum. \loclabel{family}
		\end{enumerate}
		
		\begin{proof}
			For \locimpl{weak_family}{compactifiable} we use Construction~\ref{con:from_hyperspace} on $\tuple{\F, τ}$.
				We obtain a composition of $\C$ that is compact by Observation~\ref{thm:membership_closed}.
			
			\locimpl{compactifiable}{family}.
				Let $\A(q\maps A \to B)$ be a compact composition of $\C$.
				We may suppose that $B$ is zero-dimensional by Theorem~\ref{thm:compactifiable_characterization}, that $\card{B \setminus \rng(q)} ≤ 1$ by Observation~\ref{thm:emptyset}, and that $A ⊆ [0, 1]^ω$.
				The family $\F_\A$ obtained by Construction~\ref{con:to_hyperspace} is disjoint and by Proposition~\ref{thm:compactifiable_G_delta} $G_δ$ in $\Compacta{A} ⊆ \Compacta{[0, 1]^ω}$.
				Since $\card{B \setminus \rng(q)} ≤ 1$ and the map $q$ is closed, $\FiberMap{q}\maps B \to \tuple{\F, τ_V^+}$ is a homeomorphism.
			
			\locimpl{family}{weak_family} is trivial.
		\end{proof}
	\end{theorem}
	
	\begin{question}
		Is there a similar characterization for Polishable classes?
	\end{question}
	
	Figure~\ref{fig:implications} summarizes the implications between composition-related properties and descriptive complexity of the corresponding subsets of the space of all metrizable compacta $\Compacta{[0, 1]^ω}$.
	The left part and the right part follow from the characterization theorems: \ref{thm:compactifiable_characterization}, \ref{thm:Polishable_characterization}, \ref{thm:strongly_compactifiable_characterization}, \ref{thm:strongly_Polishable_characterization}.
	The implication “compactifiable $\implies$ $G_δ$” follows from Proposition~\ref{thm:compactifiable_G_delta}.
	As a byproduct, we obtain the dashed implications.
	
	\begin{figure}[!ht] 
		\centering
		\linespread{1}
		
		\begin{tikzpicture}[
				x = {(6.5em, 0)}, 
				y = {(0, -5em)}, 
				multiline/.style = {align=center}, 
			]
			
			\begin{scope}[multiline, font=\strut]
				\node at (0, 0) (strongly_compactifiable) {strongly \\ compactifiable};
				\node at (1.7, 0) (compactifiable) {compactifiable};
				\node at (3, 0) (strongly_Polishable) {strongly \\ Polishable};
				\node at (4.7, 0) (Polishable) {Polishable};
				\node at (0, 1) (closed) {closed};
				\node at (1, 1) (F_sigma) {$F_σ$};
				\node at (3, 1) (G_delta) {$G_δ$};
				\node at (4, 1) (analytic) {analytic};
			\end{scope}
			
			\graph{
				(compactifiable) <- (strongly_compactifiable) <-> (closed) -> (F_sigma) -> (compactifiable),
				(Polishable) <- (strongly_Polishable) <-> (G_delta) <-> (analytic) -> (Polishable),
				(compactifiable) -> (G_delta),
				(compactifiable) ->[dashed] (strongly_Polishable), 
				(F_sigma) ->[dashed] (G_delta),
			};
			
			\begin{scope}[
					brace/.style = {decorate, decoration={brace, amplitude=1.2ex}},
					label/.style = {midway, font=\small},
				]
				\coordinate (start) at ([yshift=1.5ex] strongly_compactifiable.north);
				\draw[brace]
					(start) -- (Polishable |- start)
					node[label, above=0.8ex] {classes of compacta};
				
				\coordinate (start) at ([yshift=-1.8ex] closed.south);
				\draw[brace, decoration=mirror]
					(start) -- (analytic |- start)
					node[label, below=0.9ex] {existence of equivalent subsets of $\Compacta{[0, 1]^ω}$};
			\end{scope}
			
		\end{tikzpicture}
		
		\caption{Implications between the considered classes.}
		\label{fig:implications}
	\end{figure}
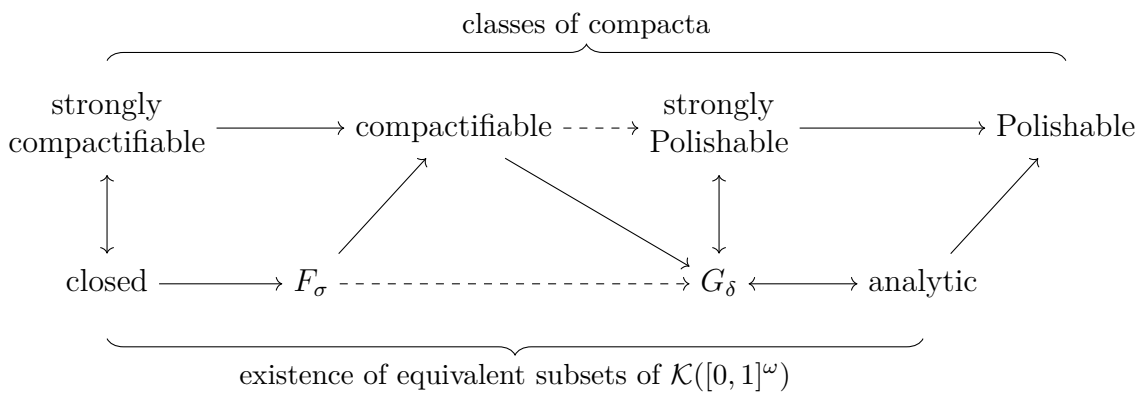
	
	\begin{question} \label{q:distinguish}
		We do not know which implications can be reversed. Namely, we have the following questions.
		\begin{enumerate}
			\item Is there a compactifiable class that is not strongly compactifiable?
			\item Is there a strongly Polishable class that is not compactifiable?
			\item Is there a Polishable class that is not strongly Polishable?
		\end{enumerate}
	\end{question}
	
	To summarize, this chapter relates (strong) compactifiability or Polishability of a class of metrizable compacta to the lowest complexity of its realizations in the hyperspace $\Compacta{[0, 1]^ω}$.
	This complexity in $\Compacta{[0, 1]^ω}$ up to the equivalence is studied in \cite{Bartos} by the first author.

\input{Wijsman}

%% file: Wijsman.tex
	\subsection{The Wijsman hypertopologies}
	
	So far we have considered mostly the hyperspace of all compact subsets $\Compacta{X}$ endowed with the Vietoris topology (or equivalently Hausdorff metric topology for metrizable $X$). There we have the one-to-one correspondence between subsets of the hyperspace and strong compositions (Construction~\ref{con:from_hyperspace} and \ref{con:to_hyperspace}).
	On the other hand, we are limited to Polishable classes of compact rather than Polish spaces.
	
	For a Polish space $X$ we would like $\Closed{X}$ to be Polish as well, but the Vietoris topology on $\Closed{X}$ is not metrizable unless $X$ is compact, and the Hausdorff metric topology is not separable unless $X$ is compact. That is why we will also consider so-called \emph{Wijsman topology}. The Wijsman topology induced by the metric $d$ is the one projectively generated by the family $\set{d(x, ⋅)\maps \Closed{X} \to ℝ}_{x ∈ X}$. It was shown in \cite{Beer} that $\Closed{X}$ with the Wijsman topology induced by a complete metric is a Polish space for a Polish space $X$.
	Usually the Wijsman topology is defined only on $\Closed{X} \setminus \set{∅}$, and is then extended to $\Closed{X}$ in a way related to the one-point compactification. For our purposes we may use the projectively generating definition directly to $\Closed{X}$, which results in $\set{∅}$ being clopen.
	
	The Wijsman topology is coarser than both Vietoris and Hausdorff metric topology, and in general they are not equal even on $\Compacta{X}$. In general, $\Compacta{X}$ is an $F_{σδ}$-subspace of $\Closed{X}$ with respect to the Wijsman topology, but it is not necessarily $G_δ$ \cite{Beer}.
	Given a metric $d$ on $X$ we may identify a set $A ∈ \Closed{X}$ with the function $d(⋅, A)\maps X \to ℝ$. Therefore, the Wijsman topology is inherited from the space of all continuous functions $\Cont(X, ℝ)$ with the topology of pointwise convergence. On the other hand, $d_H(A, B) = \sup_{x ∈ X} \abs{d(x, A) - d(x, B)}$, so the Hausdorff metric topology is inherited from $\Cont(X, ℝ)$ with the topology of uniform convergence.
	
	The Observation~\ref{thm:membership_closed} holds also for the Wijsman topologies.
	
	\begin{observation}
		If $X$ is a metrizable space and $\Closed{X}$ is endowed with a Wijsman topology, then $\R_∈ ∩ (X × \Closed{X})$ is closed in $X × \Closed{X}$.
		
		\begin{proof}
			If $F ∈ \Closed{X}$ and $x ∈ X \setminus F$, then $r := d(x, F) > 0$. We put $U = \set{y ∈ X: d(x, y) < \frac{r}{2}}$ and $\V = \set{H ∈ \Closed{X}: d(x, H) > \frac{r}{2}}$, so $U × \V$ is a neighborhood of $\tuple{x, F}$ disjoint with $\R_∈$.
		\end{proof}
	\end{observation}
	
	It follows that we may use Construction~\ref{con:from_hyperspace} also for Wijsman hyperspaces to obtain Polish compositions.
	The following proposition extends Proposition~\ref{thm:family}.
	
	\begin{proposition}
		If $X$ is a Polish space and $\Closed{X}$ is endowed with the Wijsman topology induced by a complete metric, then every analytic subset of $\Closed{X}$ is a Polishable class of Polish spaces.
	\end{proposition}
	
	\begin{remark}
		Since every Polish space can be embedded as a closed subspace to $(0, 1)^ω$, the hyperspace $\Closed{(0, 1)^ω}$ endowed with the Wijsman topology induced by a complete metric may be viewed as a Polish space of all Polish spaces.
	\end{remark}
	
	\begin{question}
		Let $\C$ be a Polishable class and let $\Closed{(0, 1)^ω}$ be endowed with a Wijsman topology induced by a complete metric.
		Does there exist an analytic (or even $G_δ$ or closed) family $\F ⊆ \Closed{(0, 1)^ω}$ such that $\F \homeo \C$?
	\end{question}
	

%% file: induced_classes.tex
\section{Induced classes} \label{sec:induced}
	
	In this section we shall analyze how the properties of being compactifiable and Polishable are preserved under various modifications and constructions of induced classes.
	
	\begin{proposition} \label{thm:countable_union}
		Strongly compactifiable, compactifiable, strongly Polishable, and Polishable classes are stable under countable unions.
		
		\begin{proof}
			For compactifiable and Polishable classes this is Observation~\ref{thm:countable_union_of_compositions}.
			Let $\C_n$, $n ∈ ω$, be strongly Polishable classes.
			By Theorem~\ref{thm:strongly_Polishable_characterization} each of them is equivalent to an analytic family $\F_n ⊆ \Compacta{[0, 1]^ω}$.
			We have $⋃_{n ∈ ω} \C_n \homeo ⋃_{n ∈ ω} \F_n$, which is also analytic and hence strongly Polishable.
			In the strongly compactifiable case we proceed analogously, but end up with an $F_σ$ family $⋃_{n ∈ ω} \F_n$.
			The conclusion follows from the non-trivial fact, that every $F_σ$ family in $\Compacta{[0, 1]^ω}$ is equivalent to a closed family, and hence is strongly compactifiable \complexityFsigma.
		\end{proof}
	\end{proposition}
	
	\begin{remark}
		In the previous proof we have used the fact that $⋃_{i ∈ I} \C_i \homeo ⋃_{i ∈ I} \D_i$ for every collection of equivalent classes $\C_i \homeo \D_i$, $i ∈ I$.
		However, it is not necessary that even $\C_i ∩ \C_j \homeo \D_i ∩ \D_j$, so we cannot use the same argument for proving preservation under intersections – compare with Proposition~\ref{thm:strongly_Polishable_intersection}.
	\end{remark}
	
	\begin{observation} \label{thm:diameter}
		Let $X$ be a metric space.
		The map $\diam\maps \powset{X} \to [0, ∞)$ is both $\tuple{τ_V^+, τ_U}$- and $\tuple{τ_V^-, τ_L}$-continuous, where $τ_U$ and $τ_L$ are the upper and lower semi-continuous topologies on $[0, ∞)$.
		It follows that $\diam$ is continuous.
		
		\begin{proof}
			If $\diam(A) < r$, then there is $ε > 0$ such that $\diam(N_ε(A)) < r$.
			Hence, $\diam(A') < r$ for every $A' ∈ N_ε(A)^+$.
			If $\diam(A) > r$, then there are points $x, y ∈ A$ and $ε > 0$ such that $d(x, y) ≥ r + 2ε$.
			Hence, $\diam(A') > r$ for every $A' ∈ B(x, ε)^- ∩ B(y, ε)^-$.
		\end{proof}
	\end{observation}
	
	\begin{corollary}
		Let $\A(q\maps A \to B)$ be a compact composition of a family $\tuple{A_b}_{b ∈ B}$.
		For every $ε > 0$ the set $B_ε := \set{b ∈ B: \diam(A_b) ≥ ε}$ is closed, and the set $B_0 := \set{b ∈ B: \diam(A_b) > 0}$ is $F_σ$.
		It follows that the corresponding families of spaces are also compactifiable.
		
		\begin{proof}
			The map $({\diam} ∘ \FiberMap{q})\maps B \to [0, ∞)$ is upper semi-continuous since $\FiberMap{q}$ is $τ_V^+$-continuous and $\diam$ is $\tuple{τ_V^+, τ_U}$-continuous by Observation~\ref{thm:diameter}.
			Note that the intervals $[ε, ∞)$ are $τ_U$-closed, and so the interval $(0, ∞)$ is $τ_U$-$F_σ$.
		\end{proof}
	\end{corollary}
	
	In definitions of many natural classes of compacta, degenerate spaces are occasionally included, resp. excluded.
	The following proposition shows that with respect to compactifiability, it does not matter.
	
	\begin{proposition} \label{thm:nondegenerate}
		If a class $\C$ of metrizable compacta is strongly compactifiable, compactifiable, strongly Polishable, or Polishable, then so are the classes $\C ∪ \set{∅}$, $\C \setminus \set{∅}$, $\C ∪ \set{1}$, and $\C_{>1}$, where $1$ denotes a one-point space and $\C_{>1}$ denotes the class of all nondegenerate members of $\C$.
		
		\begin{proof}
			The additive cases $\C ∪ \set{∅}$ and $\C ∪ \set{1}$ follow directly from Proposition~\ref{thm:countable_union}.
			The case $\C \setminus \set{∅}$ for compactifiable and Polishable classes is covered by Observation~\ref{thm:emptyset}.
			For strongly compactifiable and Polishable classes, it is easy since $\set{∅}$ is clopen in $\Compacta{[0, 1]^ω}$, and so removing it from a realization of $\C$ does not change its complexity.
			Similarly, we obtain the $\C_{>1}$ case since the degenerate sets form a closed subset of the hyperspace.
			Hence, removing degenerate spaces from a realization of $\C$ preserves the $G_δ$ complexity and turns a closed family to an $F_σ$ family (since the hyperspace is metrizable), which is enough for $\C_{>1}$ to be strongly compactifiable by Proposition~\ref{thm:countable_union}.
			
			It remains to cover the $\C_{>1}$ case for compactifiable and Polishable $\C$.
			Let $\A(q\maps A \to B)$ be a composition of $\C$ and let $C := \set{b ∈ B: \card{q\fiber{b}} > 1}$.
			On one hand, if $A$ is a metric space, then $C$ is the preimage $(\FiberMap{q})\preim{\G}$ of the family $\G := \set{K ∈ \Compacta{A}: \diam(K) > 0}$, which is $τ_V^+$-open by Observation~\ref{thm:diameter}.
			Hence, if $\A$ is a compact composition, then $q$ is closed, $\FiberMap{q}$ is $τ_V^+$-continuous, and $C$ is open and, in particular, $F_σ$, and so $\C_{>1}$ is compactifiable.
			On the other hand, $C$ is the projection of the set $\set{\tuple{a, a', b} ∈ A × A × B: q(a) = b = q(a'),\ a ≠ a'}$, which is the intersection of a closed set and an open set.
			Hence, if $\A$ is a Polish composition, then $C$ is analytic, and so $\C_{>1}$ is Polishable.
		\end{proof}
	\end{proposition}

	\begin{notation}
		Let $\C$ be a class of topological spaces.
		\begin{itemize}
			\item $\downset{\C}$ denotes the class of all subspaces of members of $\C$.
			\item $\upset{\C}$ denotes the class of all superspaces of members of $\C$.
			\item $\homeocopies{\C}$ denotes the class of all homeomorphic copies of members of $\C$.
			\item $\contimages{\C}$ denotes the class of all continuous images of members of $\C$.
			\item $\contpreimages{\C}$ denotes the class of all continuous preimages of members of $\C$, i.e. the class of all spaces than can be continuously mapped onto a member of $\C$.
		\end{itemize}
		We also denote the classes of all metrizable compacta and all continua by $\AllCompacta$ and $\AllContinua$, respectively, so we can denote e.g.\ the class of all subcontinua of members of $\C$ by $\downset{\C} ∩ \AllContinua$. For a topological space $X$ and a family $\F ⊆ \powset{X}$, the notation $\upset{\F} ∩ \powset{X}$ means “all supersets of members of $\F$ that are subsets of $X$, all endowed with the subspace topology”. This is consistent with the definition of $\upset{\C}$ above when $\powset{X}$ is viewed as a set of topological spaces.
	\end{notation}
	
	\begin{observation} \label{thm:restriction_to_continua}
		If $\C$ is a strongly compactifiable or strongly Polishable class of compacta, then so is the class $\C ∩ \AllContinua$ of all continua from $\C$ and the class $\C \setminus \AllContinua$ of all disconnected compacta from $\C$.
		If $\C$ is a strongly Polishable class of Polish spaces, then so is the class $\C ∩ \AllCompacta$ of all compacta from $\C$.
		
		\begin{proof}
			In the first case, there is a \alt{closed}{$G_δ$} family $\F ⊆ \Compacta{[0, 1]^ω}$ such that $\F \homeo \C$. We have $\C ∩ \AllContinua \homeo \F ∩ \Continua{[0, 1]^ω}$, which is \alt{closed}{$G_δ$} not only in $\Continua{[0, 1]^ω}$, but also in $\Compacta{[0, 1]^ω}$ since $\Continua{X}$ is closed in $\Compacta{X}$ for every Hausdorff space $X$.
			Similarly, $\C \setminus \AllContinua \homeo \F \setminus \Continua{[0, 1]^ω}$, which is \alt{$F_σ$}{$G_δ$} in $\Compacta{[0, 1]^ω}$.
			
			If $\C$ is a strongly Polishable class of Polish spaces, then by Construction~\ref{con:to_hyperspace} there is a Polish space $X$ and a Polish family $\F ⊆ \Closed{X}$ equivalent to $\C$ which is closed by Observation~\ref{thm:closed_family} since the hyperspace $\Closed{X}$ is Hausdorff.
			It follows that $\C ∩ \AllCompacta$ is equivalent to the family $\F ∩ \Compacta{X}$, which is closed in the Polish space $\Compacta{X}$.
		\end{proof}
	\end{observation}
	
	\begin{question}
		Is the previous observation true also for compactifiable and Polishable classes?
	\end{question}
	
	\begin{proposition}
		If $\C$ is a \alt{compactifiable}{Polishable} class, then $\downset{\C} ∩ \AllCompacta$ is a \alt{strongly compactifiable}{strongly Polishable} class.
		
		\begin{proof}
			Let $\A(q\maps A \to B)$ be a witnessing composition. It is enough to observe that $\downset{\C} ∩ \AllCompacta \homeo (\ImageMap{q})\preim{\subsets{≤1}{B}} ∩ \Compacta{A}$, which is a closed subset of $\Compacta{A}$ since the family of all degenerate subspaces of $B$, $\subsets{≤1}{B}$, is $τ_V^-$-closed in $\powset{B}$.
		\end{proof}
	\end{proposition}
	
	\begin{corollary} \label{thm:universal}
		Every hereditary class of metrizable compacta or continua with a universal element (i.e. $\C \homeo \downset{\set{X}} ∩ \AllCompacta$ or $\C \homeo \downset{\set{X}} ∩ \AllContinua$) is strongly compactifiable.
		This includes the classes of all 
			compacta, 
			totally disconnected compacta, 
			continua, 
			continua with dimension at most $n$, 
			chainable continua, 
			tree-like continua, 
			and dendrites 
			(in the realm of metrizable compacta).
	\end{corollary}

	In order to obtain a similar result for the induced class $\upset{\C} ∩ \AllCompacta{}$, we shall analyze the set $\upset{\F} ∩ \Compacta{X}$ for a family $\F ⊆ \Compacta{X}$.
	First, we shall need the following refinement of Observation~\ref{thm:membership_closed}.
	
	\begin{observation} \label{thm:inclusion_closed}
		If $X$ is a Hausdorff space, then the inclusion relation of compacts sets is closed, i.e.\ $\R_⊆ ∩ \Compacta{X}^2$ is closed in $\Compacta{X}^2$ where $\R_⊆ := \set{\tuple{A, B} ∈ \powset{X}^2: A ⊆ B}$.
		
		\begin{proof}
			If $x ∈ A \setminus B$ for some $A, B ∈ \Compacta{X}$, then there are disjoint open sets $U, V ⊆ X$ such that $x ∈ U$ and $B ⊆ V$, and hence $U^- × V^+$ is an open neighborhood of $\tuple{A, B}$ disjoint with $\R_⊆$.
		\end{proof}
	\end{observation}
	
	\begin{lemma} \labelblock{thm:compacta_projection}
		Let $X$ be a topological space.
		\begin{enumerate}
			\item The map $\K\maps \Compacta{X} \to \Compacta{\Compacta{X}}$ that maps every compact set $A ⊆ X$ to its compact hyperspace $\Compacta{A}$ is continuous. \loclabel{map}
			\item The projection $π_2\maps \R_⊆ ∩ \Compacta{X}^2 \to \Compacta{X}$ is closed and open. \loclabel{projection}
		\end{enumerate}
		
		\begin{proof}
			Let $R$ denote the relation $\R_⊆ ∩ \Compacta{X}^2$.
			Observe that for every $A ∈ \Compacta{X}$ we have $R^A = A^+ ∩ \Compacta{X} = \Compacta{A}$, which is compact.
			Hence, \locequiv{map}{projection} by Lemma~\ref{thm:clopen_projection} since $\K$ is the map $ρ$ for $R$.
			We shall prove \locref{map}.
			In fact, $\K$ is both $\tuple{τ_V^-, τ_V^-(τ_V)}$-continuous and $\tuple{τ_V^+, τ_V^+(τ_V)}$-continuous.
			(The notation $τ_V^{+/-}(τ_V)$ means $τ_V^{+/-}$ on $\Compacta{Y}$ where $Y = \Compacta{X}$ is endowed with $τ_V$.)
			
			Let $A ∈ \Compacta{X}$ and let $\V ⊆ \Compacta{X}$ be open such that $\Compacta{A} ∈ {}$\alt{$\V^-$}{$\V^+$}. To prove that $\K$ is \alt{$τ_V^-$-continuous}{$τ_V^+$-continuous} it is enough to find $\U$ a \alt{$τ_V^-$-open}{$τ_V^+$-open} neighborhood of $A$ in $\Compacta{X}$ such that $\K\im{\U} ⊆ {}$\alt{$\V^-$}{$\V^+$}.
			The set $\V$ is of the from $⋃_{i ∈ I} ⋂_{j ∈ J_i} \V_{i, j}$ where $J_i$ are finite sets and every $\V_{i, j}$ is $V^-$ or $V^+$ for some open set $V ⊆ X$.
			
			Let us start with the $τ_V^-$-continuity.
			Since $(⋃_{i ∈ I} ⋂_{j ∈ J_i} \V_{i, j})^- = ⋃_{i ∈ I} (⋂_{j ∈ J_i} \V_{i, j})^-$, we may suppose without loss of generality that $\V = ⋂_{j < m} U_j^+ ∩ ⋂_{i < n} V_i^-$ for some open sets $U_j, V_i ⊆ X$. Also, $⋂_{j < m} U_j^+ = (⋂_{j < m} U_j)^+ =: U^+$.
			We put $\U := ⋂_{i < n} (U ∩ V_i)^-$.
			Since $\Compacta{A} ∈ \V^-$, there is $B ∈ \Compacta{A} ∩ U^+ ∩ ⋂_{i < n} V_i^-$, so $B ∩ (U ∩ V_i) ≠ ∅$ for every $i < n$, and since $A ⊇ B$, we have $A ∈ \U$. 
			On the other hand, for every $B ∈ \U$ we may choose points $x_i ∈ B ∩ U ∩ V_i$ for $i < n$, and hence $\set{x_i: i < n} ∈ \Compacta{B} ∩ \V$, so $\Compacta{B} ∈ \V^-$.
			
			Now let us prove the $τ_V^+$-continuity.
			We have \[\textstyle
				\Compacta{A} ⊆ \V = ⋃_{i ∈ I} ⋂_{j ∈ J_i} \V_{i, j} = ⋂_{f ∈ ∏_{i ∈ I} J_i} ⋃_{i ∈ I} \V_{i, f(i)} = ⋂_{f ∈ F} \V_f
			\] where $F := ∏_{i ∈ I} J_i$ and $\V_f := ⋃_{i ∈ I} \V_{i, f(i)}$ for $f ∈ F$. 
			Since $\Compacta{A}$ is compact, we may suppose the sets $I$ and $F$ are finite.
			Since $(⋂_{f ∈ F} \V_f)^+ = ⋂_{f ∈ F} \V_f^+$, it is enough to find for every $f ∈ F$ an open neighborhood $\U_f$ of $A$ such that $\K\im{\U_f} ⊆ \V_f^+$. Therefore, we may suppose without loss of generality that $\V = ⋃_{i < n} U_i^+ ∪ ⋃_{j < m} V_j^-$ for some open sets $U_i, V_j ⊆ X$. Also, $⋃_{j < m} V_j^- = (⋃_{j < m} V_j)^- =: V^-$.
			
			We have $A \setminus V ∈ \Compacta{A} ⊆ \V = ⋃_{i < n} U_i^+ ∪ V^-$, and $n > 0$ since $∅ ∈ \Compacta{A} \setminus V^-$. Hence, there is some $i < n$ such that $A \setminus V ⊆ U_i$. We put $\U := (U_i ∪ V)^+$.
			We have $A = (A \setminus V) ∪ (A ∩ V) ⊆ U_i ∪ V$, so $A ∈ \U$.
			Let $B ∈ \U$. For every $C ∈ \Compacta{B}$ we have $C ⊆ B ⊆ U_i ∪ V$. Therefore, $\Compacta{B} ⊆ (U_i ∪ V)^+ ⊆ U_i^+ ∪ V^- ⊆ \V$, and so $\Compacta{B} ∈ \V^+$.
		\end{proof}
	\end{lemma}
	
	\begin{corollary} \label{thm:upset}
		Let $X$ be a topological space and $\F ⊆ \Compacta{X}$. 
		\begin{enumerate}
			\item If $\F$ is closed, then $\upset{\F} ∩ \Compacta{X}$ is closed.
			\item If $X$ is Polish and $\F$ is analytic, then $\upset{\F} ∩ \Compacta{X}$ is analytic.
		\end{enumerate}
		
		\begin{proof}
			Observe that $\upset{\F} ∩ \Compacta{X}$ is the $π_2$-image of the set $\H := \R_⊆ ∩ (\F × \Compacta{X})$.
			If $\F$ is closed, then $\H$ is closed in $\R_⊆ ∩ \Compacta{X}^2$, and the claim follows since the map $π_2\restr{\R_⊆ ∩ \Compacta{X}^2}$ is closed by Lemma~\ref{thm:compacta_projection}.
			If $\F$ is analytic, then $\H$ is analytic since $\Compacta{X}$ is Polish and $\R_⊆$ is closed in $\Compacta{X}^2$ by Observation~\ref{thm:inclusion_closed}. The claim follows since the map $π_2$ is continuous.
		\end{proof}
	\end{corollary}
	
	\begin{proposition} \label{thm:upper_class}
		If $\C$ is a strongly compactifiable or a strongly Polishable class of compacta, then so is the corresponding class of all metrizable compact superspaces $\upset{\C} ∩ \AllCompacta$.
		
		\begin{proof}
			Let us denote the Hilbert cube by $Q$ and let $Z$ be a Z-set in $Q$ that is homeomorphic to $Q$ (it exists by \cite[Lemma~5.1.3]{Mill}).
			Our class $\C$ is equivalent to a closed or an analytic family $\F ⊆ \Compacta{Z}$.
			We show that $\upset{\C} ∩ \AllCompacta$ is equivalent to $\upset{\F} ∩ \Compacta{Q}$, which is closed or analytic by Corollary~\ref{thm:upset}.
			Clearly, every member of $\upset{\F} ∩ \Compacta{Q}$ is homeomorphic to a member of $\upset{\C} ∩ \AllCompacta$.
			On the other hand, let $K ∈ \upset{\C} ∩ \AllCompacta$.
			We may suppose that $K ∈ \Compacta{Z}$.
			Since $K$ has a subspace $C ∈ \C$, there is a homeomorphism $h\maps C \to F ∈ \F$.
			By \cite[Theorem~5.3.7]{Mill} $h$ can be extended to a homeomorphism $\bar{h}\maps Q \to Q$.
			We have $K \homeo \bar{h}\im{K} ∈ \upset{\F} ∩ \Compacta{Q}$.
		\end{proof}
	\end{proposition}
	
	\begin{example} \label{ex:uncountable}
		The class of all uncountable metrizable compacta is strongly compactifiable.
		Since every uncountable metrizable compactum contains a copy of the Cantor space, the class is equivalent to $\upset{\set{2^ω}} ∩ \AllCompacta$.
	\end{example}

	\begin{proposition}
		If $\C$ is a strongly compactifiable or a strongly Polishable class of compacta, then so is the corresponding class of all metrizable compact continuous preimages $\contpreimages{\C} ∩ \AllCompacta$.
		
		\begin{proof}
			Let $Q$ denote the Hilbert cube $[0, 1]^ω$ and let $\F ⊆ \Compacta{Q}$ be equivalent to $\C$.
			We will show that $\contpreimages{\C} ∩ \AllCompacta \homeo \H := \set{K ∈ \Compacta{Q × Q}: π_2\im{K} ∈ \F}$.
			Clearly, $\H ⊆ \contpreimages{\F} ∩ \AllCompacta$.
			On the other hand, let $K ∈ \contpreimages{\F} ∩ \AllCompacta$.
			There is an embedding $e\maps K \into Q$, and there is a continuous map $f\maps K \onto Y ⊆ Q$ for some $Y ∈ \F$.
			The map $(e \diag f)\maps K \to Q × Q$ defined by $x \mapsto \tuple{e(x), f(x)}$ is an embedding because of the embedding $e$, so $K \homeo \rng(e \diag f) ⊆ Q × Q$.
			At the same time $π_2\im{\rng(e \diag f)} = \rng(f) = Y$, and so $\rng(e \diag f) ∈ \H$.
			Altogether, we have $\contpreimages{\C} ∩ \AllCompacta \homeo \contpreimages{\F} ∩ \AllCompacta \homeo \H$.
			Since $\H = (\ImageMap{π_2})\preim{\F}$, if $\F$ is closed or analytic, so is $\H$.
		\end{proof}
	\end{proposition}
	
	\begin{example}
		We have another way to see that the class of all disconnected metrizable compacta $\AllCompacta \setminus \AllContinua$ is strongly compactifiable (besides Observation~\ref{thm:restriction_to_continua}) since it is exactly $\contpreimages{\set{2}} ∩ \AllCompacta$, where $2$ denotes the two-point discrete space.
	\end{example}
	
	\begin{example} \label{ex:infty_components}
		The class of all metrizable compact spaces with infinitely many components is strongly compactifiable since it is exactly $\contpreimages{\set{ω + 1}} ∩ \AllCompacta$, where $ω + 1$ denotes the convergent sequence.
		
		\begin{proof}
			For every metrizable compactum $X$ we consider the equivalence $\sim$ induced by its components.
			$X/{\sim}$ may be viewed as a subspace of the Cantor space $2^ω$. If $X$ has infinitely many components, then $X/{\sim}$ contains a nontrivial converging sequence.
			The conclusion follows from the fact that every closed subspace of $2^ω$ is its retract.
		\end{proof}
	\end{example}
	
	\begin{example} \label{ex:non-Peano}
		Let $\N$ denote the class of all topological spaces that are \emph{not} locally connected.
		The class of all non-Peano metrizable continua $\N ∩ \AllContinua$ is strongly compactifiable since it is exactly $\contpreimages{\set{H}} ∩ \AllContinua$, where $H$ denotes the harmonic fan.
		The class of all non-locally connected metrizable compacta $\N ∩ \AllCompacta$ is strongly compactifiable since it is exactly $\contpreimages{\set{ω + 1, H}} ∩ \AllCompacta$.
		
		\begin{proof}
			Since Peano continua are exactly continuous images of the unit interval, every continuum that maps continuously onto $H$ (which is clearly not locally connected) is not Peano, so $\contpreimages{\set{H}} ∩ \AllContinua ⊆ \N ∩ \AllContinua$.
			On the other hand, it is known that each member of $\N ∩ \AllContinua$ maps continuously onto $H$ \cite{Bellamy}.
			
			Let $K ∈ \AllCompacta$.
			By Example~\ref{ex:infty_components}, $K$ has infinitely many components if and only if $K$ continuously maps onto $ω + 1$, and in this case $K$ is not locally connected.
			This is because $K$ contains a convergent sequence such that each its member and the limit are in different components.
			So we may suppose that $K$ has finitely many components.
			If $K ∈ \N$, then one of the components is a non-Peano continuum, and so $K ∈ \contpreimages{\set{H}}$ as before.
			On the other hand if $K ∈ \contpreimages{\set{H}}$, then one of its components maps onto a subfan $H' ⊆ H$ that contains infinitely many endpoints of $H$.
			It follows that $H' ∈ \N$, and so $K ∈ \N$.
		\end{proof}
	\end{example}
	
	\begin{question}
		Is the class of all Peano continua strongly compactifiable?
		We will show in Corollary~\ref{thm:couniversal} that it is compactifiable.
	\end{question}
	
	\begin{example} \label{ex:dendrites}
		Let $\D$ denote the class of all dendrites, $\N$ the class of all non-locally connected spaces, and $S^1$ the unit circle.
		Both $\D$ and $\AllContinua \setminus \D$ are strongly compactifiable classes – $\D$ by Corollary~\ref{thm:universal}, and $\AllContinua \setminus \D$ since dendrites are exactly Peano continua not containing a simple closed curve, so $\AllContinua \setminus \D \homeo (\upset{\set{S^1}} ∪ \N) ∩ \AllContinua$, which is strongly compactifiable by Proposition~\ref{thm:upper_class} and Example~\ref{ex:non-Peano}.
	\end{example}

	In the following paragraphs we shall prove a preservation theorem for $\contimages{\C} ∩ \AllCompacta$ and a necessary condition for being a strongly Polishable class.
	
	\begin{lemma} \label{thm:partial_graphs}
		Let $X$, $Y$ be metrizable. The following sets are $G_δ$.
		\begin{itemize}
			\item $\G_{\homeo} := \set{G ∈ \Compacta{X × Y}: G$ is a graph of a partial homeomorphism$}$,
			\item $\G_{\onto} := \set{G ∈ \Compacta{X × Y}: G$ is a graph of a partial continuous surjection$}$.
		\end{itemize}
		
		\begin{proof}
			A set $G ∈ \Compacta{X × Y}$ is a member of $\G_{\homeo}$ if and only if the maps $π_X\restr{G}$ and $π_Y\restr{G}$ are injective. The necessity is clear. On the other hand, if they are injective, then they are homeomorphisms onto their images since $G$ is compact.
			It follows that $G$ is the graph of the homeomorphism $π_Y\restr{G} ∘ (π_X\restr{G})\inv \maps π_X\im{G} \to π_Y\im{G}$.
			Analogously, $G ∈ \G_{\onto}$ if and only if $π_X\restr{G}$ is injective.
			
			For every $n ∈ ℕ$ let $\F_n := \set{F ∈ \Compacta{X × Y}: \card{π_Y\im{F}} = 1$ and $\diam(F) ≥ \frac{1}{n}}$, which is a closed set since $\ImageMap{π_Y}$ is continuous, $\subsets{1}{Y}$ is closed in $\Compacta{Y}$, and $\diam\maps \Compacta{X × Y} \to [0, ∞)$ is continuous. The map $π_Y\restr{G}$ is \emph{not} injective if and only if there are $x_1 ≠ x_2 ∈ X$ and $y ∈ Y$ such that $\set{\tuple{x_1, y}, \tuple{x_2, y}} ⊆ G$ if and only if there is $n ∈ ℕ$ and a set $F ∈ \F_n$ such that $F ⊆ G$, i.e. if and only if $G ∈ ⋃_{n ∈ ℕ} \upset{\F_n} ∩ \Compacta{X × Y}$, which is an $F_σ$ set by Corollary~\ref{thm:upset}. Analogously for $π_X\restr{G}$.
		\end{proof}
	\end{lemma}
	
	It is known that the homeomorphic classification for compact metric spaces is analytic \cite[Proposition~14.4.3]{Gao}.
	We shall use the following formulation of the result.
	
	\begin{corollary} \label{thm:homeo_relations}
		Let $X$, $Y$ be Polish spaces. The following relations are analytic.
		\begin{itemize}
			\item $\R_{\homeo} := \set{\tuple{A, B} ∈ \Compacta{X} × \Compacta{Y}: B$ is homeomorphic to $A}$,
			\item $\R_{\onto} := \set{\tuple{A, B} ∈ \Compacta{X} × \Compacta{Y}: B$ is a continuous image of $A}$.
		\end{itemize}
		
		\begin{proof}
			We have $\R_{\homeo} = \set{\tuple{π_X\im{G}, π_Y\im{G}}: G ∈ \G_{\homeo}} = (\ImageMap{π_X} \diag \ImageMap{π_Y})\im{\G_{\homeo}}$, which is a continuous image of a $G_δ$ set by Lemma~\ref{thm:partial_graphs}.
			Analogously for $\R_\onto$.
		\end{proof}
	\end{corollary}
	
	\begin{proposition}
		If $\C$ is a strongly Polishable class of compacta, then the corresponding class of all metrizable compact continuous images $\contimages{\C} ∩ \AllCompacta$ is also strongly Polishable.
		Moreover, the class $\contimages{\C} ∩ \AllContinua$ is compactifiable.
		
		\begin{proof}
			There is an analytic family $\F ⊆ \Compacta{[0, 1]^ω}$ such that $\C \homeo \F$.
			We have $\contimages{\C} ∩ \AllCompacta \homeo \contimages{\F} ∩ \Compacta{[0, 1]^ω} = \R_\onto\im{\F} = π_2\im{\H}$ where $\H = \R_\onto ∩ (\F × \Compacta{[0, 1]^ω})$, which is an analytic set by Corollary~\ref{thm:homeo_relations}.
			Moreover, either $\contimages{\C} ∩ \AllContinua$ contains $[0, 1]$ and so every Peano continuum, and hence $\contimages{\F} ∩ \Continua{[0, 1]^ω}$ is compactifiable by Corollary~\ref{thm:fiber_fixing_application}, or it consists only of degenerate spaces.
			In both cases, $\contimages{\C} ∩ \AllContinua$ is compactifiable.
		\end{proof}
	\end{proposition}
	
	We obtain a corollary dual to Corollary~\ref{thm:universal}.
	
	\begin{corollary} \label{thm:couniversal}
		Every class of metrizable \alt{compacta}{continua} closed under continuous images with a common model is \alt{strongly Polishable}{compactifiable}.
		This includes the class of all Peano continua (images of $[0, 1]$) and the class of all weakly chainable continua (images of the pseudoarc).
	\end{corollary}
	
	We finally give the necessary condition.
	
	\begin{theorem} \label{thm:analytic}
		If $\C$ is a strongly Polishable class of compacta, then $\homeocopies{\C} ∩ \Compacta{X}$ is analytic for every Polish space $X$.
		
		\begin{proof}
			There is an analytic set $\F ⊆ \Compacta{[0, 1]^ω}$ such that $\F \homeo \C$. We have $\homeocopies{\C} ∩ \Compacta{X} = \R_{\homeo}\im{\F} = π_2\im{\H}$ where $\R_{\homeo}$ is the relation of being homeomorphic on $\Compacta{[0, 1]^ω} × \Compacta{X}$ and $\H = \R_{\homeo} ∩ (\F × \Compacta{X})$, which is an analytic set by Corollary~\ref{thm:homeo_relations}.
		\end{proof}
	\end{theorem}
	
	\begin{corollary} \label{thm:analytic_embeddable}
		If $\C$ is a class of metrizable compacta embeddable into a Polish space $X$, then it is equivalent to $\homeocopies{\C} ∩ \Compacta{X}$.
		Hence, $\C$ is strongly Polishable if and only if $\homeocopies{\C} ∩ \Compacta{X}$ is analytic.
	\end{corollary}
	
	\begin{example}
		Every strongly Polishable class of zero-dimensional compacta is equivalent to an analytic family in $\Compacta{2^ω}$ by Corollary~\ref{thm:analytic_embeddable}, and if it contains a copy of $2^ω$, then it is compactifable by Corollary~\ref{thm:fiber_fixing_application}.
	\end{example}
	
	\begin{remark} \label{thm:almost_never_closed}
		For a strongly compactifiable class $\C$, the family $\homeocopies{\C} ∩ \Compacta{[0, 1]^ω}$ is almost never closed.
		In fact, this happens if and only if $\homeocopies{\C}$ is one of the countably many classes listed in \complexityClosed.
	\end{remark}
	
	\begin{example}
		By \cite[Theorem~27.5]{Kechris} the class of all uncountable compacta in $\Compacta{[0, 1]^ω}$ is analytically complete.
		Together with Example~\ref{ex:uncountable} this shows that there is a strongly compactifiable class $\C$ such that $\homeocopies{\C} ∩ \Compacta{[0, 1]^ω}$ is not Borel.
		It also follows that the class of all countable metrizable compacta is coanalytically complete, and hence is not strongly Polishable.
		Note that by a classical result of Mazurkiewicz and Sierpiński~\cite{MS}, countable metrizable compacta are exactly countable successor ordinals and zero.
	\end{example}
	
	\begin{example}
		By \cite{Marcone} the following classes are also coanalytically complete, and hence not strongly Polishable: hereditarily decomposable continua, dendroids, $λ$-dendroids, arcwise connected continua, uniquely arcwise connected continua, hereditarily locally connected continua.
	\end{example}

	Let us conclude with a result on preservation under intersections.
	
	\begin{proposition} \label{thm:strongly_Polishable_intersection}
		Let $\set{\C_n: n ∈ ω}$, $\C$, $\D$ be classes of metrizable compacta.
		\begin{enumerate}
			\item If the classes $\C_n$ are \alt{strongly Polishable}{Polishable}, then so is the class $⋂_{n ∈ ω} \homeocopies{\C_n}$.
			\item If the classes $\C$ and $\D$ are \alt{strongly Polishable}{Polishable}, then so is the class $\C ∩ \homeocopies{\D}$.
		\end{enumerate}
		
		\begin{proof}
			In the strongly Polishable case we have $⋂_{n ∈ ω} \homeocopies{\C_n} \homeo ⋂_{n ∈ ω} \homeocopies{\C_n} ∩ \Compacta{[0, 1]^ω}$, which is an analytic set by Theorem~\ref{thm:analytic}.
			
			In the Polishable case, by Theorem~\ref{thm:Polishable_characterization} for every $n ∈ ω$ there is a $G_δ$ subset $F_n ⊆ [0, 1]^ω × ω^ω$ such that $\set{F_n^x: x ∈ ω^ω} \homeo \C_n$.
			By \cite[Theorem~28.8]{Kechris} the maps $ρ_n\maps ω^ω \to \Compacta{[0, 1]^ω}$ defined by $x \mapsto F_n^x$ are Borel.
			Let $i, j ∈ ω$.
			We put $A_{i, j} := \set{\tuple{x, y} ∈ ω^ω × ω^ω: F_i^x \homeo F_j^y} = (ρ_i × ρ_j)\preim{\R_{\homeo}}$.
			Since the relation $\R_{\homeo}$ is analytic and the map $ρ_i × ρ_j$ is Borel, the set $A_{i, j}$ is analytic.
			Hence, also the set $A := \set{\tuple{x_n}_{n ∈ ω} ∈ (ω^ω)^ω: \tuple{x_i, x_j} ∈ A_{i, j}$ for every $i, j ∈ ω}$ and its projection $π_0\im{A} ⊆ ω^ω$ are analytic.
			Observe that $⋂_{n ∈ ω} \homeocopies{\C_n} \homeo \set{F_0^x: x ∈ π_0\im{A}}$, and so the intersection is Polishable by Corollary~\ref{thm:subcomposition}.
			
			Unlike $\C ∩ \D$, the class $\C ∩ \homeocopies{\D}$ is equivalent to $\homeocopies{\C} ∩ \homeocopies{\D}$, which is (strongly) Polishable by the previous claim.
		\end{proof}
	\end{proposition}
	
	\begin{remark}
		A similar argument would give us that if $\C$ is strongly compactifiable and $\homeocopies{\D} ∩ \Compacta{[0, 1]^ω}$ is closed, then $\C ∩ \homeocopies{\D}$ is strongly compactifiable, but by Remark~\ref{thm:almost_never_closed}, $\homeocopies{\D}$ would have to be one of countably many special classes.
		One of these classes is the class of all metrizable continua $\AllContinua$, so Observation~\ref{thm:restriction_to_continua} is a special case.
	\end{remark}
	
	
	\begin{example}
		We shall extend Example~\ref{ex:dendrites}.
		Let $\P$ be the class of all Peano continua.
		The class $\P \setminus \D$ is strongly Polishable by Corollary~\ref{thm:couniversal} and Proposition~\ref{thm:strongly_Polishable_intersection} since it is equivalent to $\P ∩ \upset{\set{S^1}}$.
	\end{example}

%% file: inverse_limits.tex
\section{Compactifiability and inverse limits} \label{sec:limits}
	
	In the last section we give a construction of compact or Polish compositions of classes of spaces expressible as inverse limits of sequences of spaces and bonding maps from suitable families.
	
	\medskip
	
	First, we shall recall some standard notions and the related notation.
	An \emph{inverse sequence} is a pair $\tuple{X_*, f_*}$ where $X_* = \tuple{X_n}_{n ∈ ω}$ is a sequence of topological spaces and $f_* = \tuple{f_n\maps X_n \from X_{n + 1}}_{n ∈ ω}$ is a sequence of continuous maps.
	For every $n ≤ m ∈ ω$ we denote by $f_{n, m}$ the composition $(f_n ∘ f_{n + 1} ∘ \cdots ∘ f_{m - 1})\maps X_n \from X_m$.
	In particular, $f_{n, n} = \id_{X_n}$ and $f_{n, n + 1} = f_n$ for every $n$.
	
	The \emph{limit} of $\tuple{X_*, f_*}$ is the pair $\tuple{X_∞, \tuple{f_{n, ∞}}_{n ∈ ω}}$ where
	the \emph{limit space} $X_∞$ is the subspace of $∏_{n ∈ ω} X_n$ consisting of all sequences $x_* = \tuple{x_n}_{n ∈ ω}$ such that $x_n = f_n(x_{n + 1})$ for every $n$, and the \emph{limit maps} $f_{n, ∞}\maps X_n \from X_∞$ are just the coordinate projections restricted to $X_∞$.
	Abstractly, the limit is defined by its universal property: the limit maps satisfy $f_{n, ∞} = f_n ∘ f_{n + 1, ∞}$ for every $n$, and for every other family of continuous maps $g_n\maps X_n \from Y$ satisfying $g_n = f_n ∘ g_{n + 1}$ for every $n$, there is a unique continuous map $g_∞\maps X_∞ \from Y$ such that $g_n = f_{n, ∞} ∘ g_∞$ for every $n$.
	
	\medskip
	
	Recall that a \emph{tree} is a partially ordered set $T$ with the smallest element such that for every node $t ∈ T$ the set $\set{s ∈ T: s < t}$ is well-ordered.
	A \emph{lower subset} of $T$ is a subset $S ⊆ T$ such that for every $t ≤ s ∈ T$ with $s ∈ S$ we have also $t ∈ S$.
	A \emph{subtree} of a tree $T$ is a lower subset $S ⊆ T$ endowed with the induced ordering.
	We will be interested in trees of countable height.
	These can be always represented as subtrees of $A^{<ω} = ⋃_{n ∈ ω} A^n$ for a sufficiently large set $A$.
	The members of $A^{<ω}$ are $A$-valued tuples $t$ of finite length $\len{t}$, and they are ordered by extension, i.e. $t ≤ s$ if and only if $s\restr{\len{t}} = t$.
	For $T$ a subtree of $A^{<ω}$ and $n ∈ ω$, the \emph{level $n$} of $T$, denoted by $T_n$, is the set $\set{t ∈ T: \len{t} = n} = T ∩ A^n$.
	
	Let $T$ be a tree.
	A node $s ∈ T$ is a \emph{successor} of a node $t ∈ T$ if $s > t$ and there is no other node $s > s' > t$.
	We denote this by $s \succ t$.
	A tree is \alt{\emph{countably}}{\emph{finitely}} \emph{splitting} if every node has only \alt{countably}{finitely} many successors.
	Every countably splitting tree of countable height may be realized as a subtree of $ω^{<ω}$.
	
	Let $t, s ∈ A^{<ω}$ for some $A$.
	We denote the \emph{concatenation} of the tuples $t$ and $s$ by $t \concat s$.
	That means, $t \concat s ∈ A^{<ω}$, $(t \concat s)(n) = t(n)$ for $n < \len{t}$ and $(t \concat s)(\len{t} + n) = s(n)$ for $n < \len{s}$.
	For $a ∈ A$, the notation $t \concat a$ is a shortcut for $t \concat \tuple{a}$.
	Note that for a subtree $T ⊆ A^{<ω}$, every successor of a node $t ∈ T$ is of the form $t \concat a$ for some $a ∈ A$.
	
	Let $T$ be a tree.
	Recall that a \emph{branch} of $T$ is any maximal chain $α ⊆ T$, i.e. a subset of $T$ whose elements are pairwise comparable and which is maximal with respect to inclusion.
	Suppose that $T$ is a subtree of some $A^{<ω}$.
	In that case, for every infinite branch of $T$ there is a unique sequence $α ∈ A^ω$ such that the infinite branch as a set is $\set{α\restr{n}: n ∈ ω}$.
	For this reason it is common to identify infinite branches of $A^{<ω}$ with $A^ω$.
	By $T_∞$ we denote the \emph{body} of $T$, i.e. the set of all infinite branches of $T ⊆ A^{<ω}$ viewed as a subspace of $A^ω$ with the product topology with $A$ being discrete.
	The standard basic open subsets of $T_∞$ are of the form $N_t := \set{α ∈ T_∞: α\restr{\len{t}} = t}$ for $t ∈ T$.
	It is easy to see that $T_∞$ is always a closed subspace of the space $A^ω$, and so is Polish if $A$ is countable.
	For more details on trees see for example \cite[Section~I.2]{Kechris}.
	
	\begin{definition}
		Let $T$ be a subtree of $A^{<ω}$ for some $A$.
		By a \emph{$T$-inverse system} we mean a pair $\tuple{X_*, f_*}$ where $X_* = \tuple{X_t}_{t ∈ T}$ is a family of topological spaces and $f_* = \tuple{f_{t, s}\maps X_t \from X_s}_{t ≤ s ∈ T}$ is a family of continuous maps such that $f_{t, t} = \id_{X_t}$ for every $t$ and $f_{t, s} ∘ f_{s, r} = f_{t, r}$ for every $t ≤ s ≤ r$.
		Of course, the system is determined by the successor maps $f_{t, t \concat \tuple{a}}$ where $t \concat \tuple{a} ∈ T$.
		Note that an inverse sequence may be viewed as a $1^{<ω}$-inverse system.
	\end{definition}
	
	\begin{construction} \label{con:tree}
		Let $T$ be a subtree of $ω^{<ω}$ and let $\tuple{X_*, f_*}$ be a $T$-inverse system.
		The following construction produces a composition of the limit spaces along the infinite branches of $T$.
		
		We consider the inverse sequence $\tuple{X^\oplus_*, f^\oplus_*}$ obtained by summing $\tuple{X_*, f_*}$ along each level of $T$, i.e. for each $n ∈ ω$ we put $X^\oplus_n := \TopSum_{t ∈ T_n} X_t$ and $f^\oplus_n := (\TopSum_{t ∈ T_n} f^\oplus_t)\maps X^\oplus_n \from X^\oplus_{n + 1}$ where the maps $f^\oplus_t := (\CoDiag_{s \succ t} f_{t, s})\maps X_t \from ∑_{s \succ t} X_s$ are preliminary codiagonal sums of all maps going to $X_t$.
		(We denote codiagonal sums by $\CoDiag$ and diagonal products by $\Diag$. The notation is inspired by \cite[2.1.11 and 2.3.20]{Engelking}.)
		
		Moreover, for each branch $α ∈ T_∞$ we consider the inverse sequence $\tuple{X^α_*, f^α_*}$ defined as the restriction of $\tuple{X_*, f_*}$ to $α$, i.e. $X^α_n = X_{α\restr{n}}$ and $f^α_n = f_{α\restr{n}, α\restr{n + 1}}\maps X^α_n \from X^α_{n + 1}$.
		For every $n ∈ ω$ we denote the embedding $X^α_n \into X^\oplus_n$ by $e^α_n$.
		This yields a natural transformation $e^α_*\maps \tuple{X^α_*, f^α_*} \into \tuple{X^\oplus_*, f^\oplus_*}$ and the limit embedding $e^α_∞\maps X^α_∞ \into X^\oplus_∞$.
		
		\begin{proof}[Claim] \let \qed \relax
		The family of subspaces $\tuple{\rng(e^α_∞)}_{α ∈ T_∞}$ is a decomposition of $X^\oplus_∞$, and the induced map $q\maps X^\oplus_∞ \to T_∞$ (where $q\fiber{α} = \rng(e^α_∞)$) is continuous.
		Hence, we have a composition $\A(q\maps X^\oplus_∞ \to T_∞)$ of the family of embeddings $\tuple{e^α_∞}_{α ∈ T_∞}$.
		If all spaces $X_t$ for $t ∈ T$ are Polish, then the composition is Polish.
		If all spaces $X_t$ for $t ∈ T$ are metrizable compacta and $T$ is finitely splitting, then the composition is compact.
		\end{proof}
		
		\begin{proof}
			Without loss of generality, we may suppose that $X_t ⊆ X^\oplus_n$ for every $n ∈ ω$ and $t ∈ T_n$, and that $X^α_∞ ⊆ X^\oplus_∞$ for every $α ∈ T_∞$.
			
			First, $\tuple{X^α_∞}_{α ∈ T_∞}$ is a decomposition of $X^\oplus_∞$.
			Clearly, for every $x_* ∈ X^\oplus_∞ ⊆ ∏_{n ∈ ω} X^\oplus_n$ and every $n ∈ ω$ there is a unique node $t_n ∈ T_n$ such that $x_n ∈ X_{t_n}$, and since $x_n = f^\oplus_n(x_{n + 1})$, we have that $t_{n + 1}$ is a successor of $t_n$ and $x_n = f_{t_n, t_{n + 1}}(x_{n + 1})$.
			Hence, $α := \set{t_n: n ∈ ω}$ is the unique infinite branch such that $x_n = f^α_n(x_{n + 1})$ for every $n ∈ ω$, i.e. such that $x_* ∈ X^α_∞$.
			
			Let $n ∈ ω$ and $t ∈ T_n$.
			For every $x_* ∈ X^α_∞ ⊆ X^\oplus_∞$ we have $α(n) = t$ if and only if $x_n ∈ X_t$.
			Hence, we have $q\preim{N_t} = \set{x_* ∈ X^\oplus_∞: x_n ∈ X_t} = (f^\oplus_{n, ∞})\preim{X_t}$, and $X_t$ is clopen in $X^\oplus_n$.
			Therefore, $q\maps X^\oplus_∞ \to T_∞$ is continuous.
			
			If $T$ is \alt{countably}{finitely} splitting, then every level $T_n$ is \alt{countable}{finite}, and so every space $X^\oplus_n$ is \alt{Polish}{metrizable compact} if all spaces $X_t$ are.
			So is their limit $X^\oplus_∞$ as a closed subspace of their product.
			The indexing space $T_∞$ is a closed subset of $ω^ω$, and therefore is Polish.
			Moreover, if $T$ is finitely splitting, $T_∞$ is a closed subset of $∏_{n ∈ ω} F_n$ for some finite sets $F_n ⊆ ω$ since every level $T_n$ is finite, and so it is a metrizable compactum.
		\end{proof}
	\end{construction}
	
	\begin{remark}
		Construction~\ref{con:tree} gives a way of proving that some class of spaces is compactifiable or Polishable.
		On the other hand, note that every compact composition $\A(q\maps A \to 2^ω)$ gives us a $2^{<ω}$-inverse system of inclusions.
		Namely, for every $t ∈ T := 2^{<ω}$ we put $X_t := q\preim{N_t}$,
		and for very $s ≥ t$ we define $f_{t, s}$ by the inclusion $X_s ⊆ X_t$.
		We obtain a $T$-inverse system $\tuple{X_*, f_*}$ and for every $α ∈ T_∞ = 2^ω$ we have $X^α_∞ = ⋂_{n ∈ ω} q\preim{N_{α\restr{n}}} = q\preim{⋂_{n ∈ ω} N_{α\restr{n}}} = q\fiber{α}$.
		Moreover, $X^\oplus_n = A$ for every $n$, so by applying Construction~\ref{con:tree} to $\tuple{X_*, f_*}$, we obtain the composition $\A$ we started with.
	\end{remark}
	
	\begin{definition}
		For a class $\F$ of continuous maps, we call a topological space \emph{$\F$-like} if it is the limit of an inverse sequence with bonding maps in $\F$.
		By $\Obj(\F)$ we denote the class of all domains and codomains of the maps from $\F$.
		
		For a class $\P$ of topological spaces, we call a topological space \emph{$\P$-like} if it is $\F$-like for $\F$ being the class of all continuous surjections between spaces from $\P$.
		Classically, $\set{[0, 1]}$-like spaces are called \emph{arc-like}, and $\set{S^1}$-like spaces are called \emph{circle-like}.
	\end{definition}
	
	\begin{proposition} \label{thm:tree_from_maps}
		Let $\F$ be a countable family of continuous maps.
		There is a subtree $T ⊆ ω^{<ω}$ and a $T$-inverse system $\tuple{X_*, f_*}$ such that $\set{X^α_∞: α ∈ T_∞}$ is equivalent to the class of all $\F$-like spaces.
		Moreover, we may have $T ⊆ 2^{<ω}$ if every space $X$ that is the codomain of infinitely maps from $\F$ is $\F$-like (in particular, if $\id_X ∈ \F$).
		
		\begin{proof}
			If every $\F$-like space is empty, then the empty tree or a single-branch tree with empty maps works.
			Otherwise, let us fix a nonempty $\F$-like space $X_∅$ formally distinct from each member of $\Obj(\F)$.
			Moreover, for every $X ∈ \Obj(\F)$, let us fix a constant map $c_X\maps X \to X_∅$.
			We put $\F' := \F ∪ \set{c_X: X ∈ \Obj(\F)}$.
			A space is $\F'$-like if and only if it is $\F$-like since every inverse sequence with bonding maps from $\F'$ either has all bonding maps in $\F$ or starts with some $c_X$ and continues with maps from $\F$.
			We have extended $\F$ to $\F'$ just to have a common codomain to serve as the root of our tree.
			
			Let $A := \card{\F'} ≤ ω$ and let $\tuple{f_n}_{n ∈ A}$ be an enumeration of $\F'$.
			We associate every $t ∈ A^{<ω}$ with the composition $f_{t(0)} ∘ f_{t(1)} ∘ \cdots ∘ f_{t(\len{t} - 1)}$ if the composition is possible and if the codomain is $X_∅$.
			Namely, let $T$ be the subtree of $A^{<ω} ⊆ ω^{<ω}$ consisting of all tuples $t$
				such that $\dom(f_{t(n)}) = \cod(f_{t(n + 1)})$ for every $n + 1 < \len{t}$
				and $\cod(f_{t(0)}) = X_∅$ or $t = ∅$.
			We put $X_t := \dom(f_{t(\len{t} - 1)})$ for $t ∈ T \setminus \set{∅}$.
			Note that $X_∅$ is already defined.
			For every $t \concat n ∈ T$ we put $f_{t, t \concat n} := f_n$.
			This defines the desired $T$-inverse system $\tuple{X_*, f_*}$.
			The first level consists exactly of the added maps $c_X$, i.e. $\set{f_{∅, \tuple{n}}: \tuple{n} ∈ T} = \set{c_X: X ∈ \Obj(\F)}$.
			Moreover, the restrictions $\tuple{X^α_*, f^α_*}$ along infinite branches $α ∈ T_∞$ are exactly inverse sequences with bonding maps in $\F'$ and starting at $X_∅$, which are exactly all inverse sequences with bonding maps from $\F$ prepended with the corresponding map $c_X$.
			
			Now let us turn the tree $T ⊆ ω^{<ω}$ into a tree $S ⊆ 2^{<ω}$, and define the corresponding $S$-inverse system $\tuple{Y_*, g_*}$.
			First, we define canonical transformations between $ω^{<ω}$ and $2^{<ω}$.
			For every $n ∈ ω$ let $[n]$ be the sequence of $n$ ones followed by zero,
			and for every $t ∈ ω^{<ω}$ let $φ(t)$ be the concatenation $[t(0)] \concat [t(1)] \concat \cdots \concat [t(\len{t} - 1)]$.
			This defines an injective map $φ\maps ω^{<ω} \to 2^{<ω}$.
			Essentially, each branching $t \concat 0, t \concat 1, t \concat 2, …$ is replaced by $t \concat 0, t \concat \tuple{1, 0}, t\concat \tuple{1, 1, 0}, …$
			The image $φ\im{ω^{<ω}}$ consists of all sequences ending with $0$ and the empty sequence.
			Let $ψ\maps 2^{<ω} \to ω^{<ω}$ be the extension of $φ\inv$ by $ψ(s \concat 1) := ψ(s)$ for $s ∈ 2^{<ω}$.
			
			Let $S := ψ\preim{T}$, which is the tree generated by $φ\im{T}$.
			For each $s ∈ S$ let $Y_s := X_{ψ(s)}$, $g_{s, s \concat 1} := \id_{X_{ψ(s)}}$, and $g_{s, s \concat 0} := f_{ψ(s), ψ(s \concat 0)}$.
			This defines the desired $S$-inverse system $\tuple{Y_*, g_*}$.
			Infinite branches $α ∈ T_∞$ are in a one-to-one correspondence with infinite branches $β ∈ S_∞$ with infinitely many zeroes, and the limits of the corresponding inverse sequences $\tuple{X^α_*, f^α_*}$ and $\tuple{Y^β_*, g^β_*}$ are the same – the maps $g^β_n$ with $β(n) = 0$ are exactly the maps $f^α_n$, while the maps $g^β_n$ with $β(n) = 1$ are identities.
			Note that $S_∞$ may contain also branches with only finitely many zeroes, but the corresponding inverse sequence is eventually constant $\id_X$ for some $X ∈ \Obj(\F')$, and so its limit is $X$.
			By the construction, $X = X_t$ for some $t ∈ T$ with infinitely many successors.
			Hence, $X$ is the codomain of infinitely many maps from $\F'$, so it is either the codomain of infinitely many maps from $\F$, or $X = X_∅$.
			In both cases, $X$ is $\F$-like.
		\end{proof}
	\end{proposition}
	
	\begin{proposition} \label{thm:countable_family}
		Let $\F$ be a family of continuous maps such that $\Obj(\F)$ is a countable family of metrizable compacta.
		There is a countable family $\G ⊆ \F$ such that a space is $\F$-like if and only if it is $\G$-like.
		
		\begin{proof}
			For every $X, Y ∈ \Obj(\F)$ let $\F(X, Y)$ denote the family of all maps $f ∈ \F$ such that $f\maps X \to Y$.
			Every $\F(X, Y)$ is a subspace of the space of all continuous maps $X \to Y$ with the topology of uniform convergence, which is separable and metrizable since $X$ and $Y$ are metrizable compacta, and hence $\F(X, Y)$ is also separable.
			Let $\G(X, Y) ⊆ \F(X, Y)$ be a countable dense subset and let $\G$ be the countable family $⋃_{X, Y ∈ \Obj(\F)} \G(X, Y)$.
			
			Clearly, every $\G$-like space is $\F$-like.
			On the other hand, by Brown's approximation theorem \cite[Theorem~3]{Brown}, for every inverse sequence $\tuple{X_*, f_*}$ with bonding maps from $\F$ and fixed metrics on the spaces $X_n$, there is a sequence of numbers $ε_n > 0$ and a sequence of maps $g_n ∈ \G(X_{n + 1}, X_n)$ such that $d(f_n, g_n) < ε_n$ for every $n$ and such that the limit space of $\tuple{X_*, g_*}$ is homeomorphic to the limit space of $\tuple{X_*, f_*}$.
			Therefore, every $\F$-like space is $\G$-like.
		\end{proof}
	\end{proposition}
	
	Now we combine the previous propositions into the following theorem.
	
	\begin{theorem} \label{thm:F-like}
		Let $\F$ be a family of continuous maps.
		\begin{enumerate}
			\item If $\F$ is countable and $\Obj(\F)$ is a class of Polish spaces, then the class of all $\F$-like spaces is Polishable.
			\item If $\Obj(\F)$ is a countable family of metrizable compacta such that every $X ∈ \Obj(\F)$ is $\F$-like (in particular if $\id_X ∈ \F$), then the class of all $\F$-like spaces is compactifiable.
		\end{enumerate}
		
		\begin{proof}
			By Proposition~\ref{thm:countable_family} we may suppose that $\F$ is countable also in the compact case.
			Using Proposition~\ref{thm:tree_from_maps} we build a tree $T ⊆ ω^{<ω}$ and a $T$-inverse system such that $\F$-like spaces are exactly the limit spaces along the branches.
			Moreover, in the compact space our tree can be made finitely splitting.
			Finally, we build a \alt{Polish}{compact} composition of the class of all $\F$-like spaces using Construction~\ref{con:tree}.
		\end{proof}
	\end{theorem}
	
	\begin{corollary}
		For a countable family $\P$ of metrizable compacta, the class of all $\P$-like spaces is compactifiable.
	\end{corollary}
	
	\begin{remark}
		The class of all arc-like continua is strongly compactifiable by Corollary~\ref{thm:universal} since there is a universal arc-like continuum.
		Theorem~\ref{thm:F-like} gives another way to prove that the class of all arc-like continua is compactifiable.
		In fact, Construction~\ref{con:tree} is based on \cite[Theorem~12.22]{Nadler}, where a universal arc-like continuum is constructed.
		The difference is that in \cite[Theorem~12.22]{Nadler} all spaces $X_t$ are copies of the unit interval, the spaces $X^\oplus_n$ are extended to bigger arcs $A_n$, and the surjections $f^\oplus_n\maps X^\oplus_n \from X^\oplus_{n + 1}$ are continuously extended to surjections $g\maps A_n \from A_{n + 1}$, so we get an arc-like continuum $A_∞ ⊇ X^\oplus_∞$ as limit.
		However, such extension cannot be done with circles.
		In fact, there is no universal circle-like continuum (Observation~\ref{thm:no_universal_circle-like}).
		Yet, by Theorem~\ref{thm:F-like} the class of all circle-like continua is compactifiable.
		Because of this and Corollary~\ref{thm:universal}, a compact composition may be viewed as a weaker form of a universal element.
	\end{remark}
	
	\begin{observation} \label{thm:no_universal_circle-like}
		There is no universal circle-like continuum.
		
		\begin{proof}
			Let $\tuple{X_*, f_*}$ be an inverse sequence of circles and continuous surjections.
			We will show that if $S^1 ⊆ X_∞$, then already $S^1 = X_∞$, so $X_∞$ cannot be universal.
			
			Let us divide $S^1$ into four quarter-arcs $A_k := \set{e^{ix}: x ∈ [k \frac{π}{2}, (k + 1) \frac{π}{2}]}$, $k ∈ \set{0, 1, 2, 3}$.
			There is $n$ such that $f_{n, ∞}\im{A_0} ∩ f_{n, ∞}\im{A_2} = ∅ = f_{n, ∞}\im{A_1} ∩ f_{n, ∞}\im{A_3}$.
			Necessarily, the same condition holds for every $f_{m, ∞}$ where $m ≥ n$.
			We have that $f_{n, ∞}\restr{S^1}$ is onto.
			Otherwise, $A := f_{n, ∞}\im{S^1}$ is an arc, $f_{n, ∞}\im{A_0}$ and $f_{n, ∞}\im{A_2}$ are its disjoint subcontinua, and no two subarcs of $A$ meeting both $f_{n, ∞}\im{A_0}$ and $f_{n, ∞}\im{A_2}$ are disjoint, which is a contradiction with disjointness of $f_{n, ∞}\im{A_1}$ and $f_{n, ∞}\im{A_3}$.
			
			We have shown that $f_{m, ∞}\restr{S^1}$ is onto for every $m ≥ n$.
			But for $x ∈ X_∞ \setminus S^1$ there is $m ≥ n$ such that $f_{m, ∞}(x) ∉ f_{m, ∞}\im{S^1} = X_m$, which is a contradiction.
		\end{proof}
	\end{observation}
	
	We wonder if the constructions from this chapter may be modified to obtain strong compact or strong Polish compositions.
	In particular, we have the following question.
	
	\begin{question}
		Is the class of all circle-like continua strongly compactifiable?
	\end{question}